\documentclass[11pt,a4paper]{article}
\usepackage[T1]{fontenc}
\usepackage[utf8]{inputenc}
\usepackage{lmodern}
\usepackage[margin=2.5cm]{geometry}
\usepackage{amsmath,amssymb,amsthm}
\usepackage{enumitem}
\usepackage[hidelinks]{hyperref}
\newcommand{\subiso}{%
	\mathrel{\vtop{\offinterlineskip\ialign{##\cr
				$\subset$\cr
				\noalign{\kern 0.2ex}
				$\sim$\cr}}}
}

\theoremstyle{plain}
\newtheorem{theorem}{Theorem}
\newtheorem{proposition}{Proposition}
\newtheorem{lemma}{Lemma}

\theoremstyle{remark}
\newtheorem{example}{Example}
\newtheorem{remark}{Remark}
\newtheorem{corollary}{Corollary}

\theoremstyle{definition}
\newtheorem{definition}{Definition}

\title{Symmetric Bases and $\ell^p$ Embeddings in Musielak--Orlicz Sequence Spaces under $\delta_2$}
\author{Valentin Georgiev \and
Boyan Zlatanov}
\date{}

\begin{document}
\maketitle
\begin{center}
\small
Faculty of Mathematics and Informatics\\
University of Plovdiv ``Paisii Hilendarski''\\
24 Tzar Asen, Plovdiv 4000, Bulgaria\\[4pt]
\texttt{valentingeorgiev@uni-plovdiv.bg}\\
\texttt{bobbyz@uni-plovdiv.bg}
\end{center}

\begin{abstract}
We study isomorphic embeddings into Musielak--Orlicz sequence spaces \(\ell_\Phi\), where the normalized generating sequence \(\Phi=\{\phi_n\}\) satisfies the \(\delta_2\) condition. Using admissible normalized blocks, we construct an associated limiting family \(\mathcal C_{\Phi,\tau}\) of Orlicz functions. If a Banach space \(X\) with a symmetric basis embeds into \(\ell_\Phi\), then there is a function \(\psi\in\mathcal C_{\Phi,\tau}\) such that the basis of \(X\) is equivalent to the canonical basis of \(\ell_\psi\). An Orlicz sequence space \(\ell_\psi\) embeds if and only if \(\psi\) is equivalent near zero to a member of \(\mathcal C_{\Phi,\tau}\); in particular, \(\ell^p\) embeds exactly when such a member is equivalent to \(t^p\). Assuming differentiability of the generators on a common interval near zero, we also derive the necessary bound \(p\leq b_{\Phi,(c_n),\delta}\), where the upper growth index incorporates the summable errors in the \(\delta_2\) condition. The same index yields cotype \(q\) for every \(q\geq2\) with \(q>b_{\Phi,(c_n),\delta}\). Examples with nonuniform growth and lacunary weights illustrate this bound. For alternating Nakano exponents, the explicit limiting family gives a sharper restriction than the index alone. For general Nakano exponents, we recover the known exact cotype criterion, including attainment of the optimal exponent. For weighted Orlicz spaces, we express the cotype criterion through summable scalar errors and recover the formula using the global upper growth index when \(w\in\mathcal A\).
\end{abstract}



\section{Introduction}
The structure of subspaces of Orlicz sequence spaces has long been studied through their unit-vector bases and the growth of the defining functions. The work of Lindenstrauss and Tzafriri \cite{LindenstraussTzafriri71} and Lindberg \cite{Lindberg73} provides classical points of departure for this question. Woo \cite{Woo73} developed the theory of modular sequence spaces, while Yamamuro \cite{Yamamuro54} investigated growth properties of their generating functions. For Musielak--Orlicz sequence spaces, Katirtzoglou \cite{Katirtzoglou98} related growth conditions to type and cotype and treated, among other cases, Nakano and weighted Orlicz spaces.

The isomorphic embedding problem is particularly sensitive to the distinction between a copy of \(\ell^p\) and a coordinate or lattice copy. Peirats and Ruiz \cite{Peiratz92} studied \(\ell^p\)-copies in Musielak--Orlicz sequence spaces using block sequences and limiting families of normalized dilations. In the weighted setting, Hern\'andez and Peirats \cite{HernandezPeirats88} examined weighted sequence subspaces of Orlicz function spaces isomorphic to \(\ell^p\), whereas Fuentes and Hern\'andez \cite{FuentesHernandez88} investigated the geometry of weighted Orlicz sequence spaces and the role of the weights. Zlatanov \cite{Zlatanov07} subsequently gave a sufficient condition for an isomorphic \(\ell^p\)-embedding under a different set of hypotheses. These results motivate embedding criteria and numerical necessary bounds that can be checked directly on a given sequence of generators.

Variable-exponent spaces provide a related perspective. Hern\'andez and Ruiz \cite{HernandezRuiz12} studied the \(\ell^q\)-structure of Nakano function spaces, including lattice-isomorphic copies. More recently, Ansorena and Bello \cite{AnsorenaBello26} investigated subsymmetric basic sequences and complemented copies in variable-exponent Lebesgue spaces. These settings are related to, but have different embedding requirements from, the arbitrary isomorphic copies considered here. Musielak--Orlicz, Nakano, and weighted Orlicz sequence spaces also arise in other geometric investigations; for example, Georgiev, Zhelinski, and Zlatanov \cite{GeorgievZhelinskiZlatanov26} studied a perturbation method in these spaces.

For the embedding results, we assume the \(\delta_2\) condition for the specified Musielak--Orlicz generators and introduce a limiting family built from admissible normalized blocks. We show that a Banach space with a symmetric basis embedded in \(\ell_\Phi\) admits an Orlicz representation from this family. Conversely, every member of the family generates an Orlicz sequence space isomorphic to a block subspace of \(\ell_\Phi\). This yields a necessary and sufficient condition for an Orlicz sequence space to embed into \(\ell_\Phi\), formulated through equivalence at zero to a member of \(\mathcal C_{\Phi,\tau}\). We then obtain a necessary condition for embedded \(\ell^p\)-spaces in terms of an upper growth index that incorporates the summable errors allowed by \(\delta_2\). The differential estimate defining this index also gives a power growth condition with summable errors. Combined with the cotype characterization in \cite[Theorem~9(a)]{Katirtzoglou98}, it yields cotype \(q\) whenever \(q\geq2\) and \(q>b_{\Phi,(c_n),\delta}\). The weighted and alternating-exponent examples illustrate what the family and the index detect. For general Nakano spaces, we recover the scalar cotype criterion of Blasco and Gregori \cite[Theorem~4.3]{BlascoGregori01} by computing the least summable errors in the power growth condition. This also identifies when the optimal cotype is attained. We also give an explicit form of the weighted cotype criterion and compute it for power functions and for a function whose cotype changes with the weights. The global growth characterization for \(w\in\mathcal A\) is treated as a consequence of the known cotype theory; weighted cotype was also studied by Maleev and Zlatanov \cite{MaleevZlatanov00}. We distinguish the results proved here from the characterizations recalled from \cite{Peiratz92} and from the previously established summable approximation method of \cite{FuentesHernandez88}.

\section{Preliminaries}\label{sec:preliminaries}

We use \(\mathbb R_+=[0,\infty)\). The space \(C[0,1]\) consists of the real-valued continuous functions on \([0,1]\), equipped with the uniform norm \(\|f\|_\infty=\sup_{0\leq t\leq1}|f(t)|\).

Throughout the paper, the symbols \(M\), \(\psi\), \(\varphi\), and \(\theta\) denote individual Orlicz functions. Upper-case symbols such as \(\Phi=\{\phi_n\}\) and \(\Psi=\{\psi_n\}\) denote Musielak--Orlicz functions. Accordingly, \(\ell_\psi\) denotes an Orlicz sequence space, whereas \(\ell_\Phi\) denotes a Musielak--Orlicz sequence space.

We use parentheses for sequences of scalars and for sequences of vectors in Banach spaces, in particular in Musielak--Orlicz sequence spaces, writing \((a_n)_{n=1}^{\infty}\) and \((x_n)_{n=1}^{\infty}\), respectively. Sequences of Orlicz functions are written with braces, as in \(\Phi=\{\phi_n\}_{n=1}^{\infty}\). When the index range is clear and no ambiguity can arise, we use the shorter notation \((a_n)\), \((x_n)\), and \(\{\phi_n\}\).

For Banach spaces \(X\) and \(Y\), we write \(X\subiso Y\) when \(X\) is isomorphic to a closed subspace of \(Y\), and \(X\cong Y\) when the two spaces are isomorphic.

Following \cite{Lindestrauss96}, a sequence \((x_n)\) in a Banach space \(X\) is a Schauder basis of \(X\) if every \(x\in X\) has a unique norm-convergent expansion \(x=\sum_{n=1}^{\infty}a_nx_n\). It is a basic sequence if it is a Schauder basis of its closed linear span \([x_n]\). A basis is normalized if \(\|x_n\|=1\) for every \(n\). Two basic sequences \((x_n)\) and \((y_n)\) are equivalent if, for every scalar sequence \((a_n)\), the series \(\sum_{n=1}^{\infty}a_nx_n\) converges if and only if \(\sum_{n=1}^{\infty}a_ny_n\) converges.

Following \cite{Lindestrauss96}, a block basic sequence of \((x_n)\) is a sequence of nonzero vectors \(u_j=\sum_{n=p_{j-1}+1}^{p_j}a_nx_n\), where \(0=p_0<p_1<\cdots\) are integers and \((a_n)\) are scalars. A basic sequence \((x_n)\) is symmetric if, for every permutation \(\pi\) of \(\mathbb N\), \((x_{\pi(n)})\) is equivalent to \((x_n)\), in accordance with \cite{Lindestrauss96}.

For a subsequence \((w_{n_k})\), the indices \(n_k\) are strictly increasing, and \(w_{n_k}\to0\) means \(\lim_{k\to\infty}w_{n_k}=0\); no monotonicity is implied. We write \(w_{n_k}\uparrow\infty\) when \(w_{n_{k+1}}\geq w_{n_k}\) for all \(k\) and \(\lim_{k\to\infty}w_{n_k}=+\infty\), and \(w_{n_k}\downarrow0\) when \(w_{n_{k+1}}\leq w_{n_k}\) for all \(k\) and \(\lim_{k\to\infty}w_{n_k}=0\). The same conventions apply to \(w_n\) and to other indexed sequences such as \(v_k\), with the limit taken as the displayed index tends to infinity.

\subsection{Orlicz and Musielak--Orlicz sequence spaces}

Following \cite{Katirtzoglou98}, a function $M:\mathbb R_+\to\mathbb R_+$ is called an Orlicz function if it is continuous and convex, $M(0)=0$, and $M(x)>0$ for every $x>0$. A sequence $\Phi=\{\phi_n\}_{n=1}^{\infty}$ of Orlicz functions is called a Musielak--Orlicz function. Its modular is \(\rho_\Phi(x)=\sum_{n=1}^{\infty}\phi_n(|x_n|).\) The associated Musielak--Orlicz sequence space is
\[
\ell_\Phi = \left\{x=(x_n):\rho_\Phi(\lambda x)<\infty \text{ for some }\lambda>0\right\},
\]
endowed with the Luxemburg norm \(\|x\|_\Phi = \inf\left\{r>0: \rho_\Phi(x/r)\leq1\right\}.\) The corresponding Musielak--Orlicz heart is
\[
h_\Phi = \left\{x=(x_n):\rho_\Phi(\lambda x)<\infty \text{ for every }\lambda>0\right\}.
\]

If $\phi_n=M$ for every $n$, then $\ell_\Phi$ is the Orlicz sequence space $\ell_M$.

The complementary function of an Orlicz function \(M\) is its Fenchel conjugate, \(M^*(y)=\sup_{t\geq0}\{ty-M(t)\}\) for \(y\geq0\). It is understood as a generalized Young function and may take the value \(+\infty\).

\subsection{Growth conditions and equivalence}

An individual Orlicz function \(M\) satisfies the \(\Delta_2\) condition at zero if there exist constants \(K_M>0\) and \(t_0>0\) such that \(M(2t)\leq K_M M(t)\) for every \(0\leq t\leq t_0\). It satisfies the global \(\Delta_2\) condition if there exists \(K_M>0\) such that this inequality holds for every \(t\geq0\). Throughout the paper, an unqualified reference to the \(\Delta_2\) condition for an individual Orlicz function means the global condition; the local condition is explicitly referred to as the \(\Delta_2\) condition at zero.

An Orlicz function \(M\) satisfies the \(\Delta_2\) condition at infinity if there exist constants \(K_M>0\) and \(t_0>0\) such that \(M(2t)\leq K_M M(t)\) for every \(t\geq t_0\). This is the growth condition used in \cite[Theorem~3 and Proposition~6]{FuentesHernandez88} for summable weights satisfying the tail condition recalled in \eqref{eq:FH-weight-condition} below. It is distinct from the condition at zero and from the global condition defined above.

A Musielak--Orlicz function $\Phi=\{\phi_n\}$ satisfies the condition $\delta_2$ if there exist $K,\delta>0$ and a nonnegative sequence $(c_n)\in \ell^1$ such that, for every $n\in\mathbb N$ and $x\geq0$, the inequality $\phi_n(x)\leq\delta$ implies
\begin{equation}\label{eq:delta2}
	\phi_n(2x) \leq K \phi_n(x) + c_n.
\end{equation}

For the finite-valued Orlicz functions considered here, the standard characterization is \(\ell_\Phi=h_\Phi\) if and only if \(\Phi\in\delta_2\); see \cite{Katirtzoglou98}.

For individual Orlicz functions \(M\) and \(N\), we write \(M\sim N\) if there exist \(a,b,t_0>0\) such that \(M(at)\leq N(t)\) and \(N(bt)\leq M(t)\) for every \(0\leq t\leq t_0\). This is equivalence at zero. Equivalently, the canonical unit-vector sequences of \(\ell_M\) and \(\ell_N\) are equivalent.

The compactness and block arguments use the given generators and their right derivatives. Differentiability is used in Theorem~\ref{th:upper-index} and Corollary~\ref{cor:cotype-upper-index}.

\subsection{Cotype and a power growth condition}\label{sec:cotype-preliminaries}

Following \cite[p.~658]{BlascoGregori01}, a Banach space \(X\) has cotype \(q\), where \(2\leq q<\infty\), if there exists \(C>0\) such that, for every \(N\in\mathbb N\) and every finite sequence \((x_j)_{j=1}^{N}\) in \(X\),
\[
\left(\sum_{j=1}^{N}\|x_j\|^q\right)^{1/q}
\leq C\int_0^1\left\|\sum_{j=1}^{N}r_j(t)x_j\right\|\,dt.
\]
Here \(r_j(t)=\operatorname{sgn}(\sin(2^j\pi t))\), \(j\geq1\), are the Rademacher functions on \([0,1]\); their values at the zeros of the sine may be chosen arbitrarily. The constant \(C\) is independent of \(N\) and of the vectors. We write
\[
\operatorname{cot}(X)
=\inf\{q\in[2,\infty):X\text{ has cotype }q\},
\]
with \(\inf\varnothing=\infty\). This infimum need not be attained.

Following \cite{Katirtzoglou98}, for \(1<r<\infty\) we write \(\Phi\in\delta^r\) if there exist \(H\geq1\), \(\eta>0\), and a nonnegative sequence \((d_n)\in\ell^1\) such that, for every \(n\in\mathbb N\), \(t\geq0\), and \(\lambda\geq1\),
\begin{equation}\label{eq:delta-power-cotype}
\phi_n(\lambda t)\leq\eta
\quad\Longrightarrow\quad
\phi_n(\lambda t)\leq H\lambda^r\bigl(\phi_n(t)+d_n\bigr).
\end{equation}
The threshold in this condition is imposed on \(\phi_n(\lambda t)\). By \cite[Theorem~9(a)]{Katirtzoglou98}, for \(2\leq q<\infty\),
\begin{equation}\label{eq:cotype-characterization}
\ell_\Phi\text{ has cotype }q
\quad\Longleftrightarrow\quad
\Phi\in\delta^q.
\end{equation}
\par

\subsection{Classical limiting families}

Let \(\phi\) be an Orlicz function satisfying the global \(\Delta_2\) condition. The following subsets of \(C[0,1]\) will be used to describe the limiting behavior of normalized dilations of \(\phi\):
\[
\begin{aligned}
F_{\phi,\lambda}&=\left\{\frac{\phi(sx)}{\phi(s)}:0<s\leq\lambda\right\},
& E_{\phi,\lambda}&=\overline{F_{\phi,\lambda}},\\
E_\phi&=\bigcap_{\lambda>0}E_{\phi,\lambda},
& E_\phi(0,\infty)&=\overline{\bigcup_{\lambda>0}E_{\phi,\lambda}},\\
C_\phi&=\overline{\operatorname{co}}E_\phi,
& C_\phi(0,\infty)&=\overline{\operatorname{co}}E_\phi(0,\infty).
\end{aligned}
\]
where $\operatorname{co}$ is the convex hull. With the exception of $F_{\phi, \lambda}$, all of them are compact subsets of $C[0,1]$.

For Musielak--Orlicz functions, we use the following families from \cite{Peiratz92}. Let $(B_\alpha)_{\alpha\in I}$ be a family of subsets of $C[0,1]$. We denote by $\Gamma((B_\alpha)_{\alpha\in I})$ the set of all finite convex sums
\(\sum_{i=1}^N a_i f_{\alpha_i}\), for \(a_i>0\) and \(\sum_{i=1}^N a_i=1\),
where the indices $\alpha_1,\ldots,\alpha_N$ are distinct and $f_{\alpha_i}\in B_{\alpha_i}$. 

For a normalized Musielak--Orlicz function $\Phi=\{\phi_n\}$, that is, $\phi_n(1)=1$, set
\[
F_{\phi_n,a} = \left\{\frac{\phi_n(sx)}{\phi_n(s)}: 0<s\leq a\right\}, \qquad a>0,
\]
and define the classical limiting family
\begin{equation}\label{eq:classical-limiting-family}
	C_{\Phi,a}
	=
	\bigcap_{n=1}^{\infty}
	\overline{\Gamma((F_{\phi_j,a})_{j=n}^{\infty})},
\end{equation}
where the closures are taken in $C[0,1]$. For an individual Orlicz function $M$, the symbols $C_M$ and $C_M(0,\infty)$ denote the corresponding classical limiting families at zero and on the whole positive half-line, respectively; see \cite{Peiratz92}.

\subsection{Weighted Orlicz spaces}

Let $w=(w_n)$ be a sequence of positive weights. The weighted Orlicz sequence space and its heart are, respectively,
\[
\ell_M(w) = \left\{x=(x_n): \sum_{n=1}^{\infty}w_nM(\lambda|x_n|)<\infty \text{ for some }\lambda>0\right\}
\]
and
\[
h_M(w) = \left\{x=(x_n): \sum_{n=1}^{\infty}w_nM(\lambda|x_n|)<\infty \text{ for every }\lambda>0\right\}.
\]

It is generated by the Musielak--Orlicz function
\begin{equation}\label{eq:weighted-musielak-function}
	\phi_n^w(x)=w_nM(x).
\end{equation}

For the limiting-family arguments, we use the normalization appearing in the proof of \cite[Proposition~5]{FuentesHernandez88}. The inverse \(M^{-1}\) is the ordinary inverse on \(\mathbb R_+\): if \(0<s<t\), convexity and \(M(0)=0\) give \(M(s)\leq(s/t)M(t)<M(t)\), while \(M(t)\geq tM(1)\) for \(t\geq1\). Thus \(M\) is continuous, strictly increasing, and unbounded. Set \(s_n=M^{-1}(1/w_n)\) and define
\begin{equation}\label{eq:normalized-weighted-function}
	\widehat\phi_n^w(t)=w_nM(s_nt).
\end{equation}

Then \(\widehat\phi_n^w(1)=1\). The diagonal map \(D\), defined by \((Dx)_n=x_n/s_n\), is a linear isometry from \(\ell_M(w)\) onto \(\ell_{\widehat\Phi^w}\), where \(\widehat\Phi^w=\{\widehat\phi_n^w\}\). Indeed, for every scalar sequence \(x\) and every \(r>0\), \(\rho_{\widehat\Phi^w}(Dx/r) =\sum_{n=1}^{\infty}w_nM\left(\frac{|x_n|}{r}\right).\)

Following Hern\'andez \cite{Hernandez85}, we say that a positive weight sequence \(w=(w_n)\) belongs to the class \(\mathcal A\) if it has a subsequence \((w_{n_k})\) such that \(w_{n_k}\to0\) and \(\sum_{k=1}^{\infty}w_{n_k}=\infty\). The same class is denoted by \(\Lambda\) in \cite[Section~``Methods'']{Zlatanov13}, where this definition is also attributed to \cite{Hernandez85}.

For the monotone regimes \(w_n\uparrow\infty\) and \(w_n\downarrow0\), we recall from Peirats and Ruiz \cite{Peiratz92} the following identities for the classical limiting family associated with \(\widehat\Phi^w\):
\begin{equation}\label{eq:C1equivCinfty}
	C_{\widehat\Phi^w,1}=C_M
	\quad\text{if }w_n\uparrow\infty,
	\qquad
	C_{\widehat\Phi^w,1}=C_M(0,\infty)
	\quad\text{if }w_n\downarrow0.
\end{equation}

These identities also apply to coordinate subsequences. More precisely, let \(v=(v_k)\), where \(v_k=w_{n_k}\) for a strictly increasing sequence \((n_k)\). Then \(C_{\widehat\Phi^v,1}=C_M\) if \(v_k\uparrow\infty\), whereas \(C_{\widehat\Phi^v,1}=C_M(0,\infty)\) if \(v_k\downarrow0\). The space \(\ell_M(v)\) is isometric to the coordinate subspace of \(\ell_M(w)\) supported on \(\{n_k:k\in\mathbb N\}\), and the definitions yield \({C_{\widehat\Phi^v,1}\subseteq C_{\widehat\Phi^w,1}}\). In particular, the existence of a subsequence \(w_{n_k}\uparrow\infty\) alone does not imply \({C_{\widehat\Phi^w,1}=C_M}\). 

The universality property for weights tending to zero with divergent sum goes back to the work of Hern\'andez \cite{Hernandez81}; see the explicit attribution in \cite[p.~587]{FuentesHernandez88}. For \(w\in\mathcal A\), applying this property to a subsequence \(z_k=w_{n_k}\) with \(z_k\to0\) and \(\sum_kz_k=\infty\), and then using the corresponding coordinate inclusion into \(\ell_M(w)\), gives, for every positive weight sequence \(v\),
\begin{equation}\label{eq:universality-A}
	\ell_M(v)\subiso \ell_M(w).
\end{equation}

For summable weights, Fuentes and Hern\'andez \cite[Proposition~1, Theorem~3, and Proposition~6]{FuentesHernandez88} consider the class specified by
\begin{equation}\label{eq:FH-weight-condition}
    \sum_{n=1}^{\infty}w_n<\infty,
    \qquad
    \liminf_{n\to\infty}\frac{\sum_{k>n}w_k}{w_n}>0.
\end{equation}
Condition~\eqref{eq:FH-weight-condition} will be used in the comparison with their results and in the discussion of Example~\ref{ex:lacunary-weighted}.
\par

The weight conditions above play different roles. Convergence \(w_n\to0\), even monotonically, does not determine whether \(\sum_nw_n\) converges or diverges. Under the additional assumption \(w_n\to0\), one has \(w\in\mathcal A\) if and only if \(\sum_nw_n=\infty\): in the divergent case the whole sequence supplies the required subsequence, whereas every subsequence of a positive summable sequence is summable. In general, membership in \(\mathcal A\) requires convergence to zero and a divergent sum only along a subsequence. It implies \(\sum_nw_n=\infty\), but does not require the whole sequence to tend to zero or to be monotone.

For example, \(w_n=1/n\) decreases to zero and belongs to \(\mathcal A\). The sequence \(w_n=2^{-n}\) also decreases to zero, but is summable and therefore does not belong to \(\mathcal A\). It satisfies \eqref{eq:FH-weight-condition}, since \(\sum_{k>n}2^{-k}=2^{-n}\). The lacunary sequence \(w_n=2^{-2^n}\) is likewise decreasing and summable, but \((\sum_{k>n}w_k)/w_n\to0\), as verified in Example~\ref{ex:lacunary-weighted}; hence it lies outside \(\mathcal A\) and fails \eqref{eq:FH-weight-condition}.

For a positive summable sequence, the tail requirement in \eqref{eq:FH-weight-condition} means that there exist \(c>0\) and \(N\in\mathbb N\) such that \(\sum_{k>n}w_k\geq c w_n\) for all \(n\geq N\). Thus it requires the sum of the subsequent weights to remain at least a fixed positive fraction of the current weight. This is an additional restriction beyond summability and is not a consequence of monotone convergence to zero.

Finally, \(w_n\uparrow\infty\) excludes membership in \(\mathcal A\), but the existence of a subsequence increasing to infinity does not. Indeed, the sequence defined by \(w_{2k}=1/k\) and \(w_{2k-1}=k\) belongs to \(\mathcal A\) through its even coordinates and has an odd subsequence increasing to infinity. Accordingly, the monotone regimes in \eqref{eq:C1equivCinfty}, the universality hypothesis in \eqref{eq:universality-A}, and the summable-weight condition \eqref{eq:FH-weight-condition} must be kept distinct.
\par

For \(w\in\mathcal A\), the weighted generators \(\Phi^w=\{w_nM\}\) satisfy \(\Phi^w\in\delta_2\) if and only if \(M\) satisfies the global \(\Delta_2\) condition; see \cite{Katirtzoglou98}. Theorem~\ref{th:weighted-Orlicz-characterization} below assumes both \(M\in\Delta_2\) and \(w\in\mathcal A\). The former gives \(\widehat\Phi^w\in\delta_2\) and is the growth hypothesis in Theorem~\ref{th:LMinfty}; the latter gives the universality property~\eqref{eq:universality-A}. No growth condition on \(M^*\) is required.

	\begin{theorem}[{\cite[Theorem~2.7(i)]{Peiratz92}}]\label{th:weighted-Orlicz-characterization}
		Let \(M\) be an Orlicz function satisfying the global \(\Delta_2\) condition, let \(w\in\mathcal A\), and let \(\psi\) be an Orlicz function. Then the following assertions are equivalent:
		\begin{enumerate}[label=(\roman*)]
			\item \(\ell_\psi\subiso\ell_M(w)\);
			\item there exists \(\theta\in C_M(0,\infty)\) such that \(\psi\sim\theta\).
		\end{enumerate}
	\end{theorem}
	
	Taking \(\psi(t)=t^p\) gives the following corollary.
	
	\begin{corollary}\label{cor:weighted-lp-characterization}
		Under the assumptions of Theorem~\ref{th:weighted-Orlicz-characterization}, let \(1\leq p<\infty\). Then the following assertions are equivalent:
		\begin{enumerate}[label=(\roman*)]
			\item \(\ell^p\subiso\ell_M(w)\);
			\item there exists \(\theta\in C_M(0,\infty)\) such that \(x^p\sim\theta\).
		\end{enumerate}
	\end{corollary}
	
	\subsection{Other Preliminary Results}
	
	We also use the following standard block-basis reduction.
	
	\begin{proposition}[\cite{Peiratz92}]\label{prop:peiratz1}
		Let $X$ be a locally bounded space with a Schauder basis $(e_n)$. If $(d_n)$ is a symmetric basic sequence in $X$, then $(d_n)$ is equivalent to a block basic sequence of $(e_n)$.
	\end{proposition}
	
	For a measurable function \(f\) on \((0,\infty)\), set \(\rho_M(f)=\int_0^\infty M(|f(t)|)\,dt\). The Orlicz function space \(L_M(0,\infty)\) consists of equivalence classes, modulo equality almost everywhere, of measurable functions \(f\) such that \(\rho_M(f/r)<\infty\) for some \(r>0\). It is equipped with the Luxemburg norm \(\|f\|_M=\inf\{r>0:\rho_M(f/r)\leq1\}\).
	
	\begin{theorem}[{\cite[Theorem~1.3(i)]{Peiratz92}}]\label{th:LMinfty}
		Let \(M\) be an Orlicz function satisfying the global \(\Delta_2\) condition, and let \(X\) be a Banach space with a normalized symmetric basis \((e_i)\). Then \((e_i)\) is equivalent to a sequence of mutually disjoint functions in \(L_M(0,\infty)\) if and only if there exists \(\phi\in C_M(0,\infty)\) such that \(X\) is isomorphic to \(\ell_\phi\).
	\end{theorem}
	
	\subsection{A known embedding result for Nakano spaces}\label{sec:known-nakano}
	
	We recall a special case of the embedding theorem of Peirats and Ruiz for Nakano sequence spaces. Let \(1\leq p<q<\infty\), and define
	\[
	p_n
	=
	\begin{cases}
		p, & n \text{ odd},\\
		q, & n \text{ even}.
	\end{cases}
	\]
	
	The corresponding Nakano sequence space is
	\[
	\ell^{(p_n)} = \left\{ x=(x_n): \sum_{n=1}^{\infty} |\lambda x_n|^{p_n}<\infty \text{ for some }\lambda>0 \right\},
	\]
	equipped with the Luxemburg norm generated by the modular \({\rho_{(p_n)}(x) = \sum_{n=1}^{\infty}|x_n|^{p_n}}.\) It is the Musielak--Orlicz sequence space associated with \(\Phi=\{\phi_n\},\phi_n(t)=t^{p_n}.\) Each \(\phi_n\) is normalized, since \(\phi_n(1)=1\). Moreover, \(\phi_n(2t) = 2^{p_n}\phi_n(t) \leq 2^q\phi_n(t),\) so \(\Phi\) satisfies the \(\delta_2\) condition with \(K=2^q\) and \(c_n=0\) for every \(n\).
	
	\begin{proposition}[{\cite[Theorem~2.9]{Peiratz92}}]\label{prop:nakano-two-powers}
		Let \(1\leq p<q<\infty\), and let \(\ell^{(p_n)}\) be the Nakano sequence space defined above. Then, for \(1\leq s<\infty\), \(\ell^s\subiso\ell^{(p_n)}\Longleftrightarrow s\in\{p,q\}.\)
	\end{proposition}
	
	\subsection{Admissible normalized blocks}\label{sec:admissible-blocks}
	
	Since every Musielak--Orlicz sequence space can be represented, after a diagonal change of coordinates, by a normalized Musielak--Orlicz function, for the construction below we assume that \(\Phi=\{\phi_n\}\) is normalized, that is, \(\phi_n(1)=1\) for every \(n\in\mathbb N\). We also assume that \(\Phi\) satisfies the \(\delta_2\) condition. Fix constants \(K\geq2\), \(\delta>0\), and a nonnegative sequence \((c_n)\in\ell^1\) such that \(\phi_n(2x)\leq K\phi_n(x)+c_n\) whenever \(\phi_n(x)\leq\delta\). The parameters and limiting family introduced below are understood with respect to this fixed choice of the \(\delta_2\)-data.
	
	Fix $m\in\mathbb N$ so large that
	\begin{equation}\label{eq:choice-tau}
		\tau=2^{-m}\leq\min\{\delta,1\},
	\end{equation}
	and put $D_0=0$ and $D_k=1+K+\cdots+K^{k-1}$ for $1\leq k\leq m$. Since $(c_n)\in \ell^1$, there is $n_0\in\mathbb N$ such that
	\begin{equation}\label{eq:choice-n0}
		D_m\sum_{j=n_0}^{\infty}c_j
		\leq\frac12\min\{\delta,1\}.
	\end{equation}
	
	We need the following modification of the family used in \cite{Peiratz92}.
	\begin{definition}
		For $n\geq n_0$, let $\mathcal B_{\Phi,\tau,n}$ be the set of all functions
		\begin{equation}\label{eq:admissible-family}
			g(x)=\frac{1}{A}\sum_{j\in E}\phi_j(\tau r_jx),
			\qquad
			A=\sum_{j\in E}\phi_j(\tau r_j),
		\end{equation}
		where $E$ is a nonempty finite subset of $\{n,n+1,\ldots\}$, $r_j>0$, and
		\begin{equation}\label{eq:admissibility}
			\sum_{j\in E}\phi_j(r_j)=1.
		\end{equation}
		
		Define
		\begin{equation}\label{eq:Cphi-tau}
			\mathcal C_{\Phi,\tau}
			=\bigcap_{n=n_0}^{\infty}
			\overline{\mathcal B_{\Phi,\tau,n}},
		\end{equation}
		where the closures are taken in the uniform norm of $C[0,1]$. The notation \(\mathcal C_{\Phi,\tau}\) distinguishes this family, defined through admissible normalized blocks, from the classical family \(C_{\Phi,a}\) in \eqref{eq:classical-limiting-family}.
	\end{definition}
	
	\section{Main Results}\label{sec:main-results}
	
	The main result represents a Banach space with a symmetric basis embedded in \(\ell_\Phi\) by a function in the admissible limiting family \(\mathcal C_{\Phi,\tau}\). It also characterizes the Orlicz sequence spaces that embed into \(\ell_\Phi\): equivalence at zero to a member of this family is both necessary and sufficient. The parameters are fixed as in Section~\ref{sec:admissible-blocks}. When an element of this family is used as an Orlicz function, it is understood to be extended to \([0,\infty)\) as described after Lemma~\ref{lm:compactness}.
	
	\begin{theorem}\label{prop:1}
		Let \(\ell_\Phi\) be a Musielak--Orlicz sequence space generated by a normalized Musielak--Orlicz function \(\Phi=\{\phi_n\}\) satisfying the \(\delta_2\) condition.
		\begin{enumerate}[label=(\alph*)]
			\item\label{th:1.1}
			If \(X\) is a Banach space with a symmetric basis and is isomorphic to a subspace of \(\ell_\Phi\), then there exists \(\psi\in \mathcal C_{\Phi,\tau}\) such that \(X\) is isomorphic to \(\ell_\psi\). More precisely, the given symmetric basis of \(X\) is equivalent to the canonical basis of \(\ell_\psi\).
			\item\label{th:1.2}
			For every Orlicz function \(\psi\), the following equivalence holds:
\[
\ell_\psi\subiso\ell_\Phi
\quad\Longleftrightarrow\quad
\text{there exists }\varphi\in\mathcal C_{\Phi,\tau}
\text{ such that }\psi\sim\varphi.
\]
In particular, for \(1\leq p<\infty\), one has \(\ell^p\subiso\ell_\Phi\) if and only if there exists \(\varphi\in\mathcal C_{\Phi,\tau}\) equivalent at zero to \(t\mapsto t^p\).
\par
		\end{enumerate}
	\end{theorem}
	
	\section{Auxiliary Results}\label{sec:auxiliary-results}
	
	We establish the bounds for the normalizing constants, compactness of the admissible limiting family, and the norm comparison needed in the proof of Theorem~\ref{prop:1}. Throughout this section, the parameters are those fixed in Section~\ref{sec:admissible-blocks}.
	
	\begin{lemma}\label{lm:normalizing-constants}
		If $g\in\mathcal B_{\Phi,\tau,n}$ and $A$ is given by~\eqref{eq:admissible-family}, then
		\begin{equation}\label{eq:A-bounds}
			a\leq A\leq\tau,
			\qquad
			a=\frac{\min\{\delta,1\}}{2K^m}>0.
		\end{equation}
	\end{lemma}
	
	\begin{proof}
		The upper estimate follows from convexity and~\eqref{eq:admissibility}:
		\[
		A=\sum_{j\in E}\phi_j(\tau r_j) \leq\tau\sum_{j\in E}\phi_j(r_j)=\tau.
		\]
		
		For the lower estimate, first suppose that for some $j\in E$ and some $0\leq k<m$ one has $\phi_j(2^k\tau r_j)>\delta$, and take the smallest such $k$. Iterating the $\delta_2$ inequality up to this first crossing gives
		\[
		\delta<K^k\phi_j(\tau r_j)+D_kc_j \leq K^m A+D_m\sum_{q=n}^{\infty}c_q.
		\]
		
		By~\eqref{eq:choice-n0}, we get $A\geq\delta/(2K^m)\geq a$.
		
		If there is no such crossing, the $\delta_2$ inequality may be iterated $m$ times for every $j\in E$. Since $2^m\tau=1$, summing the resulting estimates and using~\eqref{eq:admissibility} gives
		\[
		1\leq K^mA+D_m\sum_{j\in E}c_j \leq K^mA+D_m\sum_{q=n}^{\infty}c_q.
		\]
		
		Thus $A\geq1/(2K^m)\geq a$ in this case as well.
	\end{proof}
	
	\begin{lemma}\label{lm:compactness}
		Let \(\Phi=\{\phi_n\}\) be a normalized Musielak--Orlicz function satisfying the \(\delta_2\) condition, with the parameters chosen above. Then \(\mathcal C_{\Phi,\tau}\) is a nonempty compact subset of \(C[0,1]\). Every element of \(\mathcal C_{\Phi,\tau}\) is the restriction to \([0,1]\) of a normalized Orlicz function satisfying the \(\Delta_2\) condition at zero.
	\end{lemma}
	
	\begin{proof}
		Let \(g\in\mathcal B_{\Phi,\tau,n}\) be represented as in~\eqref{eq:admissible-family}, and put \(t_j=\tau r_j\) for \(j\in E\). Denote by \(\phi'_{j,+}\) the right derivative of the convex function \(\phi_j\). Since \(\phi_j(t_j)\leq A\leq\tau\leq\delta\), convexity and the \(\delta_2\) condition give \(t_j\phi'_{j,+}(t_j) \leq\phi_j(2t_j)-\phi_j(t_j) \leq(K-1)\phi_j(t_j)+c_j.\) Moreover, convexity implies
		\[
		|\phi_j(t_jx)-\phi_j(t_jy)| \leq t_j\phi'_{j,+}(t_j)|x-y|, \qquad x,y\in[0,1].
		\]
		
		Consequently,
		\[
		|g(x)-g(y)| \leq\frac{1}{A}\sum_{j\in E}t_j\phi'_{j,+}(t_j)|x-y| \leq\left(K-1+\frac{1}{a}\sum_{j=n_0}^{\infty}c_j\right)|x-y|.
		\]
		
		Also, $0\leq g(x)\leq1$ and $g(1)=1$. Thus the families $\mathcal B_{\Phi,\tau,n}$ are uniformly bounded and equicontinuous. Their closures are compact by the Arzel\`a--Ascoli theorem.
		Each \(\mathcal B_{\Phi,\tau,n}\) is nonempty: taking \(E=\{n\}\) and \(r_n=1\) gives the admissible function \(t\mapsto\phi_n(\tau t)/\phi_n(\tau)\). Their closures form a nested sequence of nonempty compact sets, so \(\mathcal C_{\Phi,\tau}\) is nonempty and compact.
		
		Every $f\in \mathcal C_{\Phi,\tau}$ is convex, nonnegative, and satisfies $f(0)=0$ and $f(1)=1$. For $g\in\mathcal B_{\Phi,\tau,n}$ and $0\leq x\leq1/2$, the $\delta_2$ condition gives
		\[
		g(2x)\leq Kg(x)+\frac{1}{A}\sum_{j\in E}c_j \leq Kg(x)+\frac{1}{a}\sum_{j=n}^{\infty}c_j.
		\]
		
		Fix \(n\geq n_0\). Since \(f\in\overline{\mathcal B_{\Phi,\tau,n}}\), choose a sequence \(\{g_k^{(n)}\}_{k=1}^{\infty}\) in \(\mathcal B_{\Phi,\tau,n}\) such that \(g_k^{(n)}\to f\) uniformly on \([0,1]\) as \(k\to\infty\). Applying the preceding estimate to each \(g_k^{(n)}\) and passing to this uniform limit gives \(f(2x)\leq Kf(x)+a^{-1}\sum_{j=n}^{\infty}c_j\) for \(0\leq x\leq1/2\). This inequality holds for every \(n\geq n_0\). Since \((c_j)\in\ell^1\), letting \(n\to\infty\) now yields
		\begin{equation}\label{eq:limit-doubling}
			f(2x)\leq Kf(x),\qquad 0\leq x\leq\frac12.
		\end{equation}
		
		If $0<x\leq 1$, choose $q$ such that $2^{-q}\leq x$. From~\eqref{eq:limit-doubling} and monotonicity it follows that \(1=f(1)\leq K^qf(2^{-q})\leq K^qf(x).\) Thus \(f(x)>0\) for every \(x\in(0,1]\). Let \(L\geq1\) be a Lipschitz constant of \(f\) on \([0,1]\), and extend \(f\) to \([0,\infty)\) by setting \(f(x)=1+L(x-1)+(x-1)^2\) for \(x>1\). This extension is continuous and convex, and hence is a normalized Orlicz function. Inequality~\eqref{eq:limit-doubling} shows that it satisfies the \(\Delta_2\) condition at zero.
	\end{proof}
	
	Whenever an element of \(\mathcal C_{\Phi,\tau}\) is used to generate an Orlicz sequence space, we identify it with an extension of the preceding form. The choice of extension does not affect the Luxemburg norm, since the function is normalized and its values on \([0,1]\) determine the modular unit ball.
	
	In \cite[Examples~1 and~2]{FuentesHernandez88}, summable uniform approximations of normalized generating functions by \(t\) and \(t^2\) are used to obtain weighted Orlicz sequence spaces isomorphic to \(\ell^1\) and \(\ell^2\), respectively. The following lemma formulates this norm comparison for approximation by an arbitrary normalized Orlicz function.
	
	\begin{lemma}\label{lm:summable-uniform-perturbations}
		Let \(\psi\) be a normalized Orlicz function, and let \(\Psi=\{\psi_n\}\) be a normalized Musielak--Orlicz function. Set \(\varepsilon_n=\|\psi_n-\psi\|_\infty\), where the uniform norm is taken over \([0,1]\), and assume that \(\sum_{n=1}^{\infty}\varepsilon_n<\infty\). Then the canonical unit-vector sequences of \(\ell_\psi\) and \(\ell_\Psi\) are equivalent. Consequently, the identity on finitely supported sequences extends to an isomorphism between \(h_\psi\) and \(h_\Psi\).
	\end{lemma}
	
	\begin{proof}
		Set \(E=\sum_{n=1}^{\infty}\varepsilon_n.\) Let \(x=(x_n)\) be a finitely supported sequence. Since all the functions are normalized, both Luxemburg norms dominate the supremum norm: \(\|x\|_\infty\leq\|x\|_\psi\) and \(\|x\|_\infty\leq\|x\|_\Psi\).
		
		Let \(r>\|x\|_\psi\). Then \(\sum_n\psi(|x_n|/r)\leq1\), and \(|x_n|/r\leq1\) for every \(n\). Therefore,
		\[
		\sum_{n=1}^{\infty} \psi_n\left(\frac{|x_n|}{r}\right) \leq \sum_{n=1}^{\infty} \psi\left(\frac{|x_n|}{r}\right) + \sum_{n=1}^{\infty}\varepsilon_n \leq 1+E.
		\]
		
		By convexity and the equality \(\psi_n(0)=0\), \(\psi_n\left( \frac{|x_n|}{(1+E)r} \right) \leq \frac{1}{1+E} \psi_n\left(\frac{|x_n|}{r}\right).\) Hence \(\sum_{n=1}^{\infty} \psi_n\left( \frac{|x_n|}{(1+E)r} \right) \leq1,\) and consequently \(\|x\|_\Psi \leq (1+E)\|x\|_\psi.\)
		
		Interchanging the two modulars gives \(\|x\|_\psi \leq (1+E)\|x\|_\Psi.\) Thus, for every finitely supported scalar sequence \(x\), \(\frac{1}{1+E}\|x\|_\psi \leq \|x\|_\Psi \leq (1+E)\|x\|_\psi.\) These inequalities establish the equivalence of the canonical unit-vector sequences. Since the finitely supported sequences are dense in \(h_\psi\) and \(h_\Psi\), the identity extends to an isomorphism between these subspaces.
	\end{proof}
	
	\section{Proof of the Main Result}\label{sec:proof-main-result}
	
	\begin{proof}[Proof of Theorem~\ref{prop:1}]
		\ref{th:1.1} Let \(J:X\longrightarrow\ell_\Phi\) be an isomorphic embedding, and let \((d_i)\) be a normalized symmetric basis of \(X\). Since \(\Phi\in\delta_2\), the canonical unit-vector sequence \((e_n)\) is a Schauder basis of \(\ell_\Phi\). The sequence \((Jd_i)\) is a symmetric basic sequence in \(\ell_\Phi\), equivalent to \((d_i)\).
		
		By Proposition~\ref{prop:peiratz1}, there exists a normalized block basic sequence $u_i=\sum_{j=n_i}^{n_{i+1}-1}s_je_j$, $ n_i<~n_{i+1},$ which is equivalent to \((Jd_i)\), and hence to \((d_i)\). Since \(u_i\) has finite support and \(\|u_i\|_\Phi=1\), continuity of the modular on this finite-dimensional coordinate subspace gives
		
		\begin{equation}\label{eq:block-normalization}
			\sum_{j=n_i}^{n_{i+1}-1}\phi_j(|s_j|)=1.
		\end{equation}
		
		Since \(n_i\to\infty\), after discarding finitely many initial blocks and relabeling the remaining sequence, we may assume that \(n_i\geq n_0\) for every \(i\). Since a symmetric basis is equivalent to each of its tails, this does not affect the equivalence class of the sequence. For every \(i\in\mathbb N\), set
		\[
		A_i=\sum_{j=n_i}^{n_{i+1}-1}\phi_j(\tau|s_j|), \qquad \psi_i(x)=\frac{1}{A_i} \sum_{j=n_i}^{n_{i+1}-1}\phi_j(\tau|s_j|x).
		\]
		
		Disjointness of the block supports gives, for every finitely supported scalar sequence \((x_i)\),
		\begin{equation}\label{eq:block-modular-identity}
			\rho_\Phi\left(
			\sum_i x_i\tau u_i
			\right)
			=
			\sum_i A_i\psi_i(|x_i|).
		\end{equation}
		
		For each \(i\), let \(E_i=\{j:n_i\leq j<n_{i+1},\ s_j\neq0\}\). Since \(\|u_i\|_\Phi=1\), the set \(E_i\) is nonempty. Setting \(r_j=|s_j|>0\) for \(j\in E_i\), identity~\eqref{eq:block-normalization} gives \(\sum_{j\in E_i}\phi_j(r_j)=1\). Omitting the zero terms in the definitions of \(A_i\) and \(\psi_i\), we therefore obtain \(\psi_i\in\mathcal B_{\Phi,\tau,n_i}\).
		By the proof of Lemma~\ref{lm:compactness}, the set \(\overline{\mathcal B_{\Phi,\tau,n_0}}\) is compact in \(C[0,1]\). Since \(n_i\geq n_0\), every \(\psi_i\) belongs to this set. Hence there is a subsequence \(\{\psi_{i_m}\}_{m=1}^{\infty}\) converging uniformly on \([0,1]\) to a function \(\psi\).
		
		To verify that \(\psi\in\mathcal C_{\Phi,\tau}\), fix \(n\geq n_0\). Since \(n_{i_m}\to\infty\), for all sufficiently large \(m\) we have \(n_{i_m}\geq n\), and therefore \(\psi_{i_m}\in\mathcal B_{\Phi,\tau,n_{i_m}}\subseteq\mathcal B_{\Phi,\tau,n}\). Passing to the uniform limit gives \(\psi\in\overline{\mathcal B_{\Phi,\tau,n}}\). As this holds for every \(n\geq n_0\), definition~\eqref{eq:Cphi-tau} yields \(\psi\in\mathcal C_{\Phi,\tau}\).
		
		Passing to a further subsequence and relabeling, we may assume that \(\|\psi_{i_m}-\psi\|_\infty\leq2^{-m}\) for every \(m\in\mathbb N\).
		
		Let \(\Psi_0=\{\psi_i\}\), and let \(\widetilde\Psi_0=\{A_i\psi_i\}\). By Lemma~\ref{lm:normalizing-constants}, \(a\leq A_i\leq\tau\leq1.\) Consequently, for every finitely supported scalar sequence \(x\), \(\rho_{\widetilde\Psi_0}(x) \leq \rho_{\Psi_0}(x).\) Conversely, by convexity,
		\[
		\rho_{\Psi_0}(ax) = \sum_i\psi_i(a|x_i|) \leq a\sum_i\psi_i(|x_i|) \leq \sum_i A_i\psi_i(|x_i|) = \rho_{\widetilde\Psi_0}(x).
		\]
		
		These modular inequalities imply \(\|x\|_{\widetilde\Psi_0} \leq \|x\|_{\Psi_0} \leq \frac1a\|x\|_{\widetilde\Psi_0}.\)
		
		Thus the canonical unit-vector sequences of \(\ell_{\Psi_0}\) and \(\ell_{\widetilde\Psi_0}\) are equivalent. Moreover, identity~\eqref{eq:block-modular-identity}, applied to \(x/r\) for every \(r>0\), gives \(\left\|\sum_i x_i\tau u_i\right\|_\Phi =\|x\|_{\widetilde\Psi_0}\) for every finitely supported scalar sequence \(x=(x_i)\). Hence the canonical unit-vector sequence of \(\ell_{\widetilde\Psi_0}\) is isometrically equivalent to \((\tau u_i)\), which is equivalent to \((u_i)\).
		
		For the subsequence chosen above, set \(\Psi_1=\{\psi_{i_m}\}_{m=1}^{\infty}\). Since \(\sum_m\|\psi_{i_m}-\psi\|_\infty<\infty\), Lemma~\ref{lm:summable-uniform-perturbations} yields an isomorphism \(h_{\Psi_1}\cong h_\psi\), induced by the identity on finitely supported sequences. The preceding norm estimates also show that \(h_{\Psi_1}\) is isomorphic to the closed linear span of \((u_{i_m})\).
		
		Furthermore, \(\psi\in \mathcal C_{\Phi,\tau}\), so the proof of Lemma~\ref{lm:compactness} gives \(\psi(2x)\leq K\psi(x)\) for \(0\leq x\leq1/2\). Thus \(\psi\) satisfies the \(\Delta_2\) condition at zero, and consequently \(h_\psi=\ell_\psi\). Since the symmetric basis \((d_i)\) is equivalent to each of its subsequences, we obtain \(\ell_\psi=h_\psi \cong h_{\Psi_1} \cong [u_{i_m}] \cong [d_{i_m}] \cong X.\) These equivalences also show that the given symmetric basis of \(X\) is equivalent to the canonical basis of \(\ell_\psi\).
		
		\ref{th:1.2} Necessity. Assume that \(\ell_\psi\subiso\ell_\Phi\), and let \(J:\ell_\psi\longrightarrow\ell_\Phi\) be an isomorphic embedding. The canonical unit-vector sequence \((e_i)\) is a symmetric basis of \(h_\psi\). Consequently, \((Je_i)\) is a symmetric basis of the closed subspace \(J(h_\psi)\). Applying part~\ref{th:1.1} to this subspace and this basis, we obtain a function \(\varphi\in \mathcal C_{\Phi,\tau}\) such that \((Je_i)\) is equivalent to the canonical unit-vector basis of \(\ell_\varphi\).
		
		Since \(J\) is an isomorphism onto its range, there exists \(B\geq1\) such that
		\[
		B^{-1}\left\|\sum_i a_i e_i\right\|_\psi \leq \left\|\sum_i a_i e_i\right\|_\varphi \leq B\left\|\sum_i a_i e_i\right\|_\psi
		\]
		for every finitely supported scalar sequence \((a_i)\). By the equivalence criterion for canonical unit-vector sequences in Orlicz sequence spaces, the functions \(\psi\) and \(\varphi\) are equivalent at zero. Thus \(\varphi\sim\psi\).

	Sufficiency. Suppose that \(\psi\sim\varphi\) for some \(\varphi\in\mathcal C_{\Phi,\tau}\). We first construct an isomorphic embedding of \(\ell_\varphi\) into \(\ell_\Phi\). By Lemma~\ref{lm:compactness}, the function \(\varphi\) is normalized and satisfies the \(\Delta_2\) condition at zero; in particular, \(h_\varphi=\ell_\varphi\).

Starting with \(q_1=n_0\), construct inductively a sequence \(\{g_k\}_{k=1}^{\infty}\), choosing \(g_k\in\mathcal B_{\Phi,\tau,q_k}\) with \(\|g_k-\varphi\|_\infty\leq2^{-k}\). Write its admissible representation as
\[
\begin{aligned}
g_k(t)&=\frac1{A_k}\sum_{j\in E_k}\phi_j(\tau r_{k,j}t),
& A_k&=\sum_{j\in E_k}\phi_j(\tau r_{k,j}),\\
\sum_{j\in E_k}\phi_j(r_{k,j})&=1,
& r_{k,j}&>0,
\end{aligned}
\]
where \(E_k\) is a nonempty finite subset of \(\{q_k,q_k+1,\ldots\}\). After choosing \(E_k\), take \(q_{k+1}=1+\max E_k\). The next approximation exists because \(\varphi\in\overline{\mathcal B_{\Phi,\tau,q_{k+1}}}\). Thus \(\max E_k<\min E_{k+1}\). In the remainder of this argument, each \(g_k\) is defined on all of \([0,\infty)\) by the displayed finite-sum formula, so it is a normalized Orlicz function.

Set \(u_k=\sum_{j\in E_k}r_{k,j}e_j\). The admissibility equality gives \(\rho_\Phi(u_k)=1\), and hence \(\|u_k\|_\Phi=1\). The vectors \(u_k\) form a block basic sequence with successive finite supports. Lemma~\ref{lm:normalizing-constants} gives \(a\leq A_k\leq\tau\leq1\) for every \(k\).

Let \(\Theta=\{g_k\}\) and \(\widetilde\Theta=\{A_kg_k\}\), and define \(Tx=\tau\sum_kx_ku_k\) for every finitely supported scalar sequence \(x=(x_k)\). Disjointness of the supports gives, for every \(r>0\),
\[
\rho_\Phi(Tx/r)=\sum_k A_kg_k(|x_k|/r)
=\rho_{\widetilde\Theta}(x/r).
\]
Consequently, \(\|Tx\|_\Phi=\|x\|_{\widetilde\Theta}\). Moreover, convexity and the bounds on \(A_k\) imply
\[
\rho_\Theta(ax)\leq\rho_{\widetilde\Theta}(x)
\leq\rho_\Theta(x),
\qquad
a\|x\|_\Theta\leq\|Tx\|_\Phi\leq\|x\|_\Theta.
\]

Since \(\sum_k\|g_k-\varphi\|_\infty\leq\sum_k2^{-k}=1\), the norm estimates in the proof of Lemma~\ref{lm:summable-uniform-perturbations} yield
\[
\frac12\|x\|_\varphi\leq\|x\|_\Theta\leq2\|x\|_\varphi.
\]
Combining these estimates gives
\begin{equation}\label{eq:converse-block-embedding}
\frac a2\|x\|_\varphi\leq\|Tx\|_\Phi\leq2\|x\|_\varphi
\end{equation}
for every finitely supported \(x\). These sequences are dense in \(h_\varphi=\ell_\varphi\), so \(T\) extends to a bounded linear operator from \(\ell_\varphi\) into \(\ell_\Phi\). The lower estimate in \eqref{eq:converse-block-embedding} shows that the extension is injective and has closed range, equal to the closed linear span of \((u_k)\). Hence \(\ell_\varphi\subiso\ell_\Phi\).

Finally, equivalence at zero preserves the \(\Delta_2\) condition at zero. Since \(\psi\sim\varphi\), we have \(h_\psi=\ell_\psi\), and the equivalence of the canonical basic sequences gives \(\ell_\psi\cong\ell_\varphi\). Composing this isomorphism with \(T\) proves \(\ell_\psi\subiso\ell_\Phi\). Taking \(\psi(t)=t^p\), where \(1\leq p<\infty\), in the proved equivalence gives the asserted characterization of embedded \(\ell^p\)-spaces.
\par
	\end{proof}
	
	\section{Properties of the Admissible Limiting Family}\label{sec:family-properties}
	
	The next two lemmas relate the admissible family to the classical limiting family and identify it under summable approximation by a power function.
	
	\subsection{Inclusion in the classical limiting family}
	
	\begin{lemma}\label{lm:admissible-inclusion}
		Let \(\Phi=\{\phi_n\}\) be a normalized Musielak--Orlicz function satisfying the \(\delta_2\) condition, with the parameters chosen above. Then \(\mathcal C_{\Phi,\tau}\subseteq C_{\Phi,1},\) where \(C_{\Phi,1}\) is the classical family defined in \eqref{eq:classical-limiting-family}.
	\end{lemma}
	
	\begin{proof}
		Fix \(n\geq n_0\), and let \(g\in\mathcal B_{\Phi,\tau,n}\). Write
		\[
		g(x)=\frac1A\sum_{j\in E}\phi_j(\tau r_jx), \qquad A=\sum_{j\in E}\phi_j(\tau r_j),
		\]
		where \(E\subseteq\{n,n+1,\ldots\}\) is finite and nonempty, \(r_j>0\), and \(\sum_{j\in E}\phi_j(r_j)=1\).
		
		Since \(\phi_j(r_j)\leq1=\phi_j(1)\), strict monotonicity gives \(r_j\leq1\), and hence \(0<\tau r_j\leq1\). We may therefore write
		\[
		g(x)=\sum_{j\in E} \frac{\phi_j(\tau r_j)}{A} \frac{\phi_j(\tau r_jx)}{\phi_j(\tau r_j)}.
		\]
		
		The coefficients are positive and sum to one, and the indices in \(E\) are distinct. Each normalized dilation belongs to \(F_{\phi_j,1}\). Consequently, \(g\in\Gamma((F_{\phi_j,1})_{j=n}^{\infty})\). Taking closures in \(C[0,1]\) and then intersecting over the tails yields \(\mathcal C_{\Phi,\tau}\subseteq C_{\Phi,1}\).
	\end{proof}
	
	\subsection{Summable approximation by a power function}
	
	\begin{lemma}\label{lm:power-singleton-family}
		Let \(\Phi=\{\phi_n\}\) be a normalized Musielak--Orlicz function satisfying the \(\delta_2\) condition, and choose \(\tau,n_0,a\) as in \eqref{eq:choice-tau}, \eqref{eq:choice-n0}, and Lemma~\ref{lm:normalizing-constants}. For some \(1\leq p<\infty\), set \(\varepsilon_n=\sup_{0\leq t\leq1}|\phi_n(t)-t^p|\) and assume that \(\sum_{n=1}^{\infty}\varepsilon_n<\infty\).
		
		Then \(\mathcal C_{\Phi,\tau}=\{t\mapsto t^p\}\).
	\end{lemma}
	
	\begin{proof}
		Fix \(n\geq n_0\) and \(g\in \mathcal B_{\Phi,\tau,n}\). By definition,
		\[
		g(t)=\frac1A\sum_{j\in E}\phi_j(\tau r_jt), \qquad A=\sum_{j\in E}\phi_j(\tau r_j), \qquad \sum_{j\in E}\phi_j(r_j)=1,
		\]
		where \(E\) is a finite subset of \(\{n,n+1,\ldots\}\) and \(r_j>0\). Normalization implies \(r_j\leq1\), while Lemma~\ref{lm:normalizing-constants} gives \(A\geq a>0\).
		
		For \(0\leq t\leq1\), both \(\tau r_jt\) and \(\tau r_j\) belong to \([0,1]\). Using \((\tau r_jt)^p=t^p(\tau r_j)^p\), we obtain
		\[
		\begin{aligned}
			|g(t)-t^p|
			&=\frac1A\left|\sum_{j\in E}
			\bigl(\phi_j(\tau r_jt)-t^p\phi_j(\tau r_j)\bigr)\right|\\
			&\leq\frac1A\sum_{j\in E}(1+t^p)\varepsilon_j
			\leq\frac2a\sum_{j\geq n}\varepsilon_j.
		\end{aligned}
		\]
		
		The same uniform estimate holds for every function in the closure of \(\mathcal B_{\Phi,\tau,n}\). Consequently, if \(f\in\mathcal C_{\Phi,\tau}\), then \(\|f-t^p\|_\infty\leq(2/a)\sum_{j\geq n}\varepsilon_j\) for every \(n\geq n_0\). Letting \(n\to\infty\), we conclude that \(f(t)=t^p\) on \([0,1]\). Since \(\mathcal C_{\Phi,\tau}\) is nonempty by Lemma~\ref{lm:compactness}, the assertion follows.
	\end{proof}
	
	In \cite[Examples~1 and~2]{FuentesHernandez88}, the normalized generating functions satisfy the summable approximation assumption of Lemma~\ref{lm:power-singleton-family} with \(p=1\) and \(p=2\), respectively. The resulting spaces are shown there to be isomorphic to \(\ell^1\) and \(\ell^2\), and hence are separable. By the standard separability criterion for Musielak--Orlicz sequence spaces generated by finite-valued convex functions, their normalized generating sequences satisfy \(\delta_2\). Thus Lemma~\ref{lm:power-singleton-family} identifies the corresponding admissible limiting families as \(\{t\mapsto t\}\) and \(\{t\mapsto t^2\}\). The additional conclusion here is the explicit identification of these admissible families.
	
	\section{Copies of Orlicz Sequence Spaces in Weighted Orlicz Spaces}
	\label{sec:weighted-copies}
	
	Let $M$ be an Orlicz function and let $w=(w_n)$ be a sequence of positive weights. We use the weighted generators \(\Phi^w\) and \(\widehat\Phi^w\) from Section~\ref{sec:preliminaries}.
	
	For the embedding characterization, we assume that \(\widehat\Phi^w\) satisfies the \(\delta_2\) condition. This is equivalent to requiring \(\Phi^w\in\delta_2\). Indeed, if \(s_n=M^{-1}(1/w_n)\), then \(\widehat\phi_n^w(t)=\phi_n^w(s_nt)\), and the change of variable \(u=s_nt\) transfers the defining inequalities in either direction, with the same constants and summable error sequence.
	
	For the classical characterization under the additional assumptions that \(M\) satisfies the global \(\Delta_2\) condition and \(w\in\mathcal A\), see Theorem~\ref{th:weighted-Orlicz-characterization} and Corollary~\ref{cor:weighted-lp-characterization} in Section~\ref{sec:preliminaries}. The result below gives a necessary and sufficient condition in terms of the admissible family \(\mathcal C_{\widehat\Phi^w,\tau}\).
	
	\begin{theorem}\label{th:weighted-Orlicz-copy-necessary}
		Let \(M\) be an Orlicz function and let \(w=(w_n)\) be a sequence of positive weights. Suppose that the associated Musielak--Orlicz function \(\widehat\Phi^w\) satisfies the \(\delta_2\) condition. For every Orlicz function \(\psi\), one has \(\ell_\psi\subiso\ell_M(w)\) if and only if there exists \(\varphi\in\mathcal C_{\widehat\Phi^w,\tau}\) such that \(\psi\sim\varphi\).
	\end{theorem}
	
	\begin{proof}
	By \eqref{eq:normalized-weighted-function}, the diagonal map \(D:\ell_M(w)\longrightarrow\ell_{\widehat\Phi^w}\), defined by \((Dx)_n=x_n/M^{-1}(1/w_n)\), is a linear isometry onto \(\ell_{\widehat\Phi^w}\). Composition with \(D\) or \(D^{-1}\) therefore gives
\[
\ell_\psi\subiso\ell_M(w)
\quad\Longleftrightarrow\quad
\ell_\psi\subiso\ell_{\widehat\Phi^w}.
\]
Since \(\widehat\Phi^w\) is normalized and satisfies \(\delta_2\), Theorem~\ref{prop:1}\ref{th:1.2} identifies the latter assertion with the existence of \(\varphi\in\mathcal C_{\widehat\Phi^w,\tau}\) such that \(\psi\sim\varphi\). This proves both implications.
\par
	\end{proof}
	
	Taking $\psi(x)=x^p$ gives the following consequence.
	
	\begin{corollary}\label{cor:weighted-lp-copy-necessary}
		Under the assumptions of Theorem~\ref{th:weighted-Orlicz-copy-necessary}, let \(1\leq p<\infty\). Then \(\ell^p\subiso\ell_M(w)\) if and only if there exists \(\varphi\in\mathcal C_{\widehat\Phi^w,\tau}\) equivalent at zero to \(t\mapsto t^p\).
	\end{corollary}
	
	\section{An Upper Growth Index: Embeddings and Cotype}
	\label{sec:upper-index}
	
	We next introduce an upper growth index associated with a Musielak--Orlicz function and a fixed choice of its \(\delta_2\)-data. Related growth indices were considered in \cite{Yamamuro54} under uniform growth assumptions.
	
	Let \(\Phi=\{\phi_n\}\) be a normalized Musielak--Orlicz function satisfying the \(\delta_2\) condition. Fix constants \(K\geq2\), \(0<\delta\leq1\), and a nonnegative sequence \((c_n)\in\ell^1\) such that \(\phi_n(2x)\leq K\phi_n(x)+c_n\) whenever \(\phi_n(x)\leq\delta\). These restrictions on \(K\) and \(\delta\) cause no loss of generality, since \(K\) may be increased and \(\delta\) may be decreased.
	
	\begin{lemma}\label{lm:gamma}
		Let \(\Phi=\{\phi_n\}\) be a normalized Musielak--Orlicz function, and let \(0<\delta\leq1\). Then \(\gamma_\delta = \inf_{n\in\mathbb N}\phi_n^{-1}(\delta) \geq\delta>0.\)
	\end{lemma}
	
	\begin{proof}
		Since every \(\phi_n\) is convex, \(\phi_n(0)=0\), and \(\phi_n(1)=1\), for every \(x\in[0,1]\) we have \(\phi_n(x) \leq x\phi_n(1)+(1-x)\phi_n(0) =x.\) In particular, \(\phi_n(\delta)\leq\delta\). Since \(\phi_n\) is increasing, it follows that \(\phi_n^{-1}(\delta)\geq\delta\) for every \(n\). Therefore, \(\gamma_\delta = \inf_{n\in\mathbb N}\phi_n^{-1}(\delta) \geq\delta>0.\)
	\end{proof}

By Lemma~\ref{lm:gamma}, we have \(\gamma_\delta\geq\delta>0\). Assume that every \(\phi_n\) is differentiable on \((0,\gamma_\delta]\), with the derivative at \(\gamma_\delta\) understood from the left.

	The parameters \(\tau,n_0,a\) and the family \(\mathcal C_{\Phi,\tau}\) used below are chosen from these same fixed data, as in \eqref{eq:choice-tau}, \eqref{eq:choice-n0}, and Lemma~\ref{lm:normalizing-constants}. In particular, \(\tau\leq\delta\).
	
	Fix \(0<x\leq\gamma_\delta\). Convexity gives the inequality
\[
\phi_n(2x)\geq\phi_n(x)+x\phi_n'(x).
\]

At \(x=\gamma_\delta\), the derivative in this inequality is the left derivative, as stipulated above. Also, \(x\leq\gamma_\delta\leq\phi_n^{-1}(\delta)\), so \(\phi_n(x)\leq\delta\). The chosen \(\delta_2\)-estimate therefore gives \(\phi_n(2x)\leq K\phi_n(x)+c_n\). Combining the two estimates, we obtain
\[
x\phi_n'(x)\leq\phi_n(2x)-\phi_n(x)\leq(K-1)\phi_n(x)+c_n.
\]

With respect to the fixed \(\delta_2\)-data \(K\), \(\delta\), and \((c_n)\), define
\begin{align*}
\beta_{n,\delta}
&=\sup_{0<t\leq\gamma_\delta}
\frac{t\phi_n'(t)-c_n}{\phi_n(t)},\\
b_{\Phi,(c_n),\delta}
&=\limsup_{n\to\infty}\beta_{n,\delta}.
\end{align*}
The dependence on \(\delta\) is made explicit because this parameter determines the interval \((0,\gamma_\delta]\) over which the supremum is taken. For fixed \(\Phi\), \((c_n)\), and \(\delta\), the value of the index does not depend on the choice of the admissible constant \(K\). That constant is used to verify the growth estimate and to obtain the upper bound below.
\par 
	For every \(n\) and \(0<t\leq\gamma_\delta\), the definition of \(\gamma_\delta\) gives \(t\leq\phi_n^{-1}(\delta)\), and hence \(\phi_n(t)\leq\delta\). The preceding differential estimate therefore applies throughout the interval defining \(\beta_{n,\delta}\), yielding \(\beta_{n,\delta}\leq K-1\) and \(b_{\Phi,(c_n),\delta}\leq K-1<\infty\).

	\begin{theorem}\label{th:upper-index}
		Let \(\ell_\Phi\) be a Musielak--Orlicz sequence space generated by a normalized Musielak--Orlicz function \(\Phi=\{\phi_n\}\) satisfying the \(\delta_2\) condition, and fix \(\delta_2\)-data \(K\), \(\delta\), and \((c_n)\) as above. Assume that every \(\phi_n\) is differentiable on \((0,\gamma_\delta]\), where \(\gamma_\delta=\inf_n\phi_n^{-1}(\delta)\) and the derivative at \(\gamma_\delta\) is taken from the left. If \(\ell^p\), where \(1\leq p<\infty\), is isomorphic to a subspace of \(\ell_\Phi\), then \(p\leq b_{\Phi,(c_n),\delta}.\)
	\end{theorem}
	
	\begin{proof}
		Suppose, to the contrary, that \(p>b_{\Phi,(c_n),\delta}\). Choose \(r\) such that \(\max\{b_{\Phi,(c_n),\delta},0\}<r<p.\) By the definition of \(b_{\Phi,(c_n),\delta}\), there exists \(N\in\mathbb N\) such that \(\beta_{n,\delta}<r\) for every \(n\geq N\). Hence, for every \(n\geq N\) and every \(t\in(0,\gamma_\delta]\), \(\frac{t\phi_n'(t)-c_n}{\phi_n(t)} \leq\beta_{n,\delta}<r.\) Equivalently,
		\begin{equation}\label{eq:index-differential-inequality}
			t\phi_n'(t)-r\phi_n(t)-c_n<0.
		\end{equation}
		
		Since \(\ell^p\subiso \ell_\Phi\), Theorem~\ref{prop:1} \ref{th:1.2} yields an Orlicz function \(\psi\in \mathcal C_{\Phi,\tau}\) such that \(\psi\sim x^p\). Set \(n_k=\max\{n_0,N\}+k\). Since \(\psi\in\overline{\mathcal B_{\Phi,\tau,n_k}}\) for every \(k\), choose a sequence \(\{\psi_k\}_{k=1}^{\infty}\) with \(\psi_k\in\mathcal B_{\Phi,\tau,n_k}\) such that \(\|\psi_k-\psi\|_\infty\leq2^{-k}\).
		For every \(k\), write
		\[
		\psi_k(x) = \frac{1}{A_k} \sum_{j\in E_k}\phi_j(\tau r_jx), \qquad A_k = \sum_{j\in E_k}\phi_j(\tau r_j),
		\]
		where \(E_k\) is a nonempty finite subset of \(\{n_k,n_k+1,\ldots\}\), \(r_j>0\), and \(\sum_{j\in E_k}\phi_j(r_j)=1.\) In particular, \(\phi_j(r_j)\leq1=\phi_j(1)\). Since every \(\phi_j\) is increasing, it follows that \(0<r_j\leq1\). Moreover, for \(0<x\leq1\), \(0<\tau r_jx\leq\tau\leq\delta\leq\gamma_\delta.\) Therefore, inequality~\eqref{eq:index-differential-inequality} can be applied with \(t=\tau r_jx\).
		
		Set \(d_k=(rA_k)^{-1}\sum_{j\in E_k}c_j\). Summing inequality~\eqref{eq:index-differential-inequality} with \(t=\tau r_jx\) over \(j\in E_k\), and dividing by \(A_k\), gives \(x\psi_k'(x)-r\psi_k(x)-rd_k<0,  0<x\leq1.\) Consequently,
		\[
		\frac{d}{dx}\left(\frac{\psi_k(x)+d_k}{x^r}\right) =\frac{x\psi_k'(x)-r\psi_k(x)-rd_k}{x^{r+1}}<0.
		\]
		
		Thus \(x\mapsto(\psi_k(x)+d_k)/x^r\) is decreasing on \((0,1]\).
		
		By Lemma~\ref{lm:normalizing-constants}, \(A_k\geq a>0\). Hence \(0\leq d_k \leq\frac{1}{ra}\sum_{j=n_k}^{\infty}c_j \longrightarrow0.\) Since \(\psi_k(1)=1\), monotonicity yields \(\psi_k(x)+d_k\geq(1+d_k)x^r\) for every \(x\in(0,1]\). Passing to the limit, using \(\psi_k\to\psi\) uniformly and \(d_k\to0\), we obtain \(\psi(x)\geq x^r\).
		
		On the other hand, \(\psi\sim x^p\) implies that \(\psi(x)\leq Bx^p\) for some \(B>0\) and all sufficiently small \(x>0\). Therefore, \(1\leq Bx^{p-r}\), which is impossible as \(x\to0+\), since \(p-r>0\). This contradiction proves that \(p\leq b_{\Phi,(c_n),\delta}\).
	\end{proof}
	
The same differential estimate also controls the cotype of the whole space. The following consequence uses the characterization~\eqref{eq:cotype-characterization}.

\begin{corollary}\label{cor:cotype-upper-index}
Let \(\Phi=\{\phi_n\}\) be normalized and satisfy \(\delta_2\). Fix \(\delta_2\)-data \(K\geq2\), \(0<\delta\leq1\), and \((c_n)\in\ell^1\) with \(c_n\geq0\). Assume that every \(\phi_n\) is differentiable on \((0,\gamma_\delta]\), with the derivative at \(\gamma_\delta\) taken from the left. Put \(b=b_{\Phi,(c_n),\delta}\). Then
\[
\Phi\in\delta^r\qquad\text{for every }r>\max\{1,b\}.
\]
Consequently, \(\ell_\Phi\) has cotype \(q\) for every \(q\geq2\) such that \(q>b\), and
\begin{equation}\label{eq:cotype-upper-index}
\operatorname{cot}(\ell_\Phi)
\leq\max\{2,b_{\Phi,(c_n),\delta}\}.
\end{equation}
In particular, if \(b_{\Phi,(c_n),\delta}<2\), then \(\ell_\Phi\) has cotype \(2\).
\end{corollary}

\begin{proof}
Fix \(r>\max\{1,b\}\). Choose \(N\geq n_0\) such that \(\beta_{n,\delta}<r\) for every \(n\geq N\). The definition of \(\beta_{n,\delta}\) gives
\[
t\phi_n'(t)\leq r\phi_n(t)+c_n,
\qquad n\geq N,\quad 0<t\leq\gamma_\delta.
\]
For \(t>0\) and \(\lambda\geq1\) with \(\lambda t\leq\gamma_\delta\), integrating
\[
\left(\frac{\phi_n(u)}{u^r}\right)'
\leq c_nu^{-r-1}
\]
over \([t,\lambda t]\) yields
\begin{equation}\label{eq:index-integrated-growth}
\phi_n(\lambda t)
\leq\lambda^r\phi_n(t)+\frac{c_n}{r}(\lambda^r-1)
\leq\lambda^r\left(\phi_n(t)+\frac{c_n}{r}\right).
\end{equation}

Let \(\tau,n_0,a\) be chosen from the fixed \(\delta_2\)-data as above. Applying Lemma~\ref{lm:normalizing-constants} to the admissible singleton \(E=\{n\}\), \(r_n=1\), gives \(\phi_n(\tau)\geq a\) for every \(n\geq n_0\). Set \(\eta=a/2\). If \(n\geq N\) and \(\phi_n(\lambda t)\leq\eta\), then \(\lambda t<\tau\leq\delta\leq\gamma_\delta\), so~\eqref{eq:index-integrated-growth} applies whenever \(t>0\).

Define the nonnegative summable sequence
\[
d_n^{(r)}=
\begin{cases}
\eta,&n<N,\\
c_n/r,&n\geq N.
\end{cases}
\]
For \(n<N\), the condition \(\phi_n(\lambda t)\leq\eta\) directly implies
\[
\phi_n(\lambda t)\leq\eta
\leq\lambda^r\bigl(\phi_n(t)+d_n^{(r)}\bigr).
\]
For \(n\geq N\), the same inequality follows from~\eqref{eq:index-integrated-growth}; the case \(t=0\) is immediate. Thus~\eqref{eq:delta-power-cotype} holds with \(H=1\), threshold \(\eta\), and errors \((d_n^{(r)})\), proving \(\Phi\in\delta^r\).

For every \(q\geq2\) with \(q>b\), we may take \(r=q\) and apply~\eqref{eq:cotype-characterization}. Taking the infimum over these \(q\) proves~\eqref{eq:cotype-upper-index}. If \(b<2\), the choice \(r=2\) gives cotype \(2\) itself.
\end{proof}

When \(b_{\Phi,(c_n),\delta}=2\), Corollary~\ref{cor:cotype-upper-index} gives \(\operatorname{cot}(\ell_\Phi)=2\), but does not assert cotype \(2\). More generally, when \(b_{\Phi,(c_n),\delta}\geq2\), cotype at that endpoint requires a separate verification of \(\Phi\in\delta^{b_{\Phi,(c_n),\delta}}\).
\par

	\begin{example}\label{ex:nonuniform-growth}
		We construct a normalized differentiable Musielak--Orlicz function \(\Phi=\{\phi_n\}\) to which Theorem~\ref{th:upper-index} applies, although the uniform upper growth assumption imposed in \cite{Peiratz92} fails.
		
		For \(n\geq1\), put \(\varepsilon_n=2^{-n}\), \(p_n=n+1\), \(a_n=\varepsilon_n(1-1/p_n)\), and \(d_n=1-a_n\). Define
		\[
		\phi_n(t)=
		\begin{cases}
			\displaystyle \frac{\varepsilon_n}{p_nd_n}
			\left(\frac{t}{\varepsilon_n}\right)^{p_n},
			&0\leq t\leq\varepsilon_n,\\[6pt]
			\displaystyle \frac{t-a_n}{d_n},
			&t\geq\varepsilon_n.
		\end{cases}
		\]
		
		Then \(\Phi\) satisfies the \(\delta_2\) condition with \(K=2\), \(\delta=1\), and \(c_n=a_n/d_n\). For these data, \(b_{\Phi,(c_n),1}=1\), whereas \(\sup_n\sup_{0<t\leq1}\frac{t\phi_n'(t)}{\phi_n(t)}=\infty.\) Moreover, \(\ell_\Phi\cong\ell^1\), and hence, for \({1\leq s<\infty}\), \(\ell^s\subiso\ell_\Phi\Longleftrightarrow s=1.\)
	\end{example}
	
	\begin{proof}
		At \(t=\varepsilon_n\), the two defining expressions have the same value \(\varepsilon_n/(p_nd_n)\) and the same derivative \(1/d_n\). The derivative is increasing on \((0,\varepsilon_n)\) and constant on \((\varepsilon_n,\infty)\). Thus each \(\phi_n\) is convex and continuously differentiable, with \(\phi_n(0)=0\), \(\phi_n(t)>0\) for \(t>0\), and \(\phi_n(1)=1\).
		
		For \(0<t\leq\varepsilon_n\), we have \(t\phi_n'(t)/\phi_n(t)=p_n\). Since \(p_n\to\infty\), the uniform upper growth assumption of \cite{Peiratz92} is not satisfied. In addition, \(\phi_n(2t)/\phi_n(t)=2^{p_n}\) whenever \(0<t\leq\varepsilon_n/2\), so no uniform doubling estimate without additive terms holds for this sequence.
		
		On the other hand, \(0<a_n<\varepsilon_n\leq1/2\), and therefore \(0<c_n=a_n/d_n\leq2\varepsilon_n\). In particular, \((c_n)\in\ell^1\). Direct calculation gives
		\[
		t\phi_n'(t)-\phi_n(t)=
		\begin{cases}
			(p_n-1)\phi_n(t),&0<t\leq\varepsilon_n,\\
			c_n,&t\geq\varepsilon_n.
		\end{cases}
		\]
		
		For \(0<t\leq\varepsilon_n\), the first expression is at most \((p_n-1)\phi_n(\varepsilon_n)=c_n\) and consequently,
		\[
		\left(\frac{\phi_n(t)+c_n}{t}\right)' =\frac{t\phi_n'(t)-\phi_n(t)-c_n}{t^2}\leq0.
		\]
		
		Comparing the values at \(t\) and \(2t\), we obtain \(\phi_n(2t)\leq2\phi_n(t)+c_n\) for every \(t>0\). Thus the \(\delta_2\) condition holds with \(K=2\), \(\delta=1\), and the specified sequence \((c_n)\).
		
		Since \(\phi_n^{-1}(1)=1\), we have \(\gamma_1=1\). Furthermore, \(\frac{t\phi_n'(t)-c_n}{\phi_n(t)}\leq1\) for every \(t>0\), with equality whenever \(t\geq\varepsilon_n\). Hence \(\beta_{n,1}=1\) for every \(n\), and \(b_{\Phi,(c_n),1}=1\). Theorem~\ref{th:upper-index} now shows that an embedding \(\ell^s\subiso\ell_\Phi\), with \(1\leq s<\infty\), requires \(s=1\).
		
		To verify that this bound is attained, we identify \(\ell_\Phi\) with \(\ell^1\) up to equivalent norms. Convexity and normalization give \(\phi_n(t)\leq t\) for \(0\leq t\leq1\). The monotonicity established above also yields \(\phi_n(t)+c_n\geq t(1+c_n)\) for \(0<t\leq1\). Therefore, \(0\leq t-\phi_n(t)\leq c_n, 0\leq t\leq1.\) Write \(C=\sum_n c_n\). If \(x\in\ell^1\setminus\{0\}\), then \(\sum_n\phi_n(|x_n|/\|x\|_1)\leq1\), so \(\|x\|_\Phi\leq\|x\|_1\). Conversely, let \(x\in\ell_\Phi\) and \(\lambda>\|x\|_\Phi\). Then \(\sum_n\phi_n(|x_n|/\lambda)\leq1\), and normalization implies \(|x_n|/\lambda\leq1\) for every \(n\). Thus
		\[
		\frac1\lambda\sum_n|x_n| \leq\sum_n\phi_n\left(\frac{|x_n|}{\lambda}\right)+\sum_n c_n \leq1+C.
		\]
		
		Letting \(\lambda\downarrow\|x\|_\Phi\), we obtain \(\|x\|_1\leq(1+C)\|x\|_\Phi\). Hence the identity map is an isomorphism between \(\ell_\Phi\) and \(\ell^1\). Moreover, the estimate \(0\leq t-\phi_n(t)\leq c_n\) for \(0\leq t\leq1\) implies \(\sup_{0\leq t\leq1}|\phi_n(t)-t|\leq c_n\). Since \(\sum_n c_n<\infty\) and \(\Phi\in\delta_2\), Lemma~\ref{lm:power-singleton-family} applies with \(p=1\) and yields \(\mathcal C_{\Phi,\tau}=\{t\mapsto t\}\).
	\end{proof}

	\begin{example}\label{ex:lacunary-weighted}
		Let
		\[
		M(t)=
		\begin{cases}
			t^2,&0\leq t\leq1,\\
			2t-1,&t\geq1,
		\end{cases}
		\qquad
		w_n=2^{-2^n},\quad n\geq1.
		\]
		Then \(M\) is a continuously differentiable convex Orlicz function and \(\sum_n w_n<\infty\), but \(\lim_{n\to\infty}\frac{\sum_{k>n}w_k}{w_n}=0.\) For the normalized generating sequence \(\widehat\Phi^w\), the \(\delta_2\) condition holds with \(K=2\), \(\delta=1\), and \(c_n=w_n\). With these data, \(b_{\widehat\Phi^w,(w_n),1}=1\). Moreover, \(\ell_M(w)\cong\ell^1\), and, for \(1\leq s<\infty\), \(\ell^s\subiso\ell_M(w)\Longleftrightarrow s=1.\)
	\end{example}
	
	\begin{proof}
		Since \(w_{n+j}=w_n^{2^j}\), we have
		\[
		0\leq\frac{\sum_{k>n}w_k}{w_n} =\sum_{j=1}^{\infty}w_n^{2^j-1} \leq\sum_{j=1}^{\infty}w_n^j =\frac{w_n}{1-w_n}\longrightarrow0.
		\]
		
		The summability of \((w_n)\) follows from \(w_n\leq2^{-n}\).
		
		Set \(s_n=M^{-1}(1/w_n)=(1+w_n)/(2w_n)\) and \(u_n=1/s_n=2w_n/(1+w_n)\). The normalized functions are
		\[
		\widehat\phi_n^w(t)=w_nM(s_nt)=
		\begin{cases}
			\displaystyle\frac{(1+w_n)^2}{4w_n}t^2,&0\leq t\leq u_n,\\[6pt]
			(1+w_n)t-w_n,&t\geq u_n.
		\end{cases}
		\]
		
		They are continuously differentiable and satisfy \(\widehat\phi_n^w(1)=1\). Direct calculation gives
		\[
		t(\widehat\phi_n^w)'(t)-\widehat\phi_n^w(t)=
		\begin{cases}
			\widehat\phi_n^w(t),&0<t\leq u_n,\\
			w_n,&t\geq u_n.
		\end{cases}
		\]
		
		Since \(\widehat\phi_n^w(u_n)=w_n\), this expression is at most \(w_n\) for every \(t>0\). Hence \(\left(\frac{\widehat\phi_n^w(t)+w_n}{t}\right)'\leq0,\) and comparison at \(t\) and \(2t\) yields \(\widehat\phi_n^w(2t)\leq2\widehat\phi_n^w(t)+w_n\). Thus the specified \(\delta_2\)-data are valid.
		
		For \(\delta=1\), normalization gives \(\gamma_1=1\). Furthermore, \(\frac{t(\widehat\phi_n^w)'(t)-w_n}{\widehat\phi_n^w(t)}\leq1,\) with equality for \(u_n\leq t\leq1\). Therefore \(\beta_{n,1}=1\) for every \(n\), and \(b_{\widehat\Phi^w,(w_n),1}=1\). The diagonal isometry between \(\ell_M(w)\) and \(\ell_{\widehat\Phi^w}\), together with Theorem~\ref{th:upper-index}, shows that \(\ell^s\subiso\ell_M(w)\) requires \(s=1\).
		
		It remains to verify that this bound is attained. Convexity and normalization imply \(\widehat\phi_n^w(t)\leq t\) for \(0\leq t\leq1\). The monotonicity of \((\widehat\phi_n^w(t)+w_n)/t\) also gives \(\widehat\phi_n^w(t)+w_n\geq t(1+w_n)\) on \((0,1]\). Consequently, \(0\leq t-\widehat\phi_n^w(t)\leq w_n, 0\leq t\leq1.\) Lemma~\ref{lm:summable-uniform-perturbations} yields \(h_{\widehat\Phi^w}\cong\ell^1\). Since \(\widehat\Phi^w\in\delta_2\), we have \(\ell_{\widehat\Phi^w}=h_{\widehat\Phi^w}\), and hence \(\ell_M(w)\cong\ell^1\). Lemma~\ref{lm:power-singleton-family} also give \({\mathcal C_{\widehat\Phi^w,\tau}=\{t\mapsto t\}}\).
	\end{proof}
	
\begin{remark}[Strict inclusion of the admissible family]\label{rem:strict-inclusion}
For the normalized generating sequence in Example~\ref{ex:lacunary-weighted}, the inclusion in Lemma~\ref{lm:admissible-inclusion} is strict:
\[
\mathcal C_{\widehat\Phi^w,\tau}
=\{t\mapsto t\}
\subsetneq C_{\widehat\Phi^w,1}.
\]
Indeed, \(0<u_j<1\), and for every \(0<s\leq u_j\), both \(s\) and \(st\) lie in the quadratic part of \(\widehat\phi_j^w\) whenever \(0\leq t\leq1\). Therefore,
\[
\frac{\widehat\phi_j^w(st)}{\widehat\phi_j^w(s)}=t^2,
\qquad 0\leq t\leq1.
\]
Thus \(t\mapsto t^2\) belongs to \(F_{\widehat\phi_j^w,1}\) for every \(j\). A convex combination consisting of this single function belongs to each tail family \(\Gamma((F_{\widehat\phi_j^w,1})_{j=n}^{\infty})\). Taking closures and intersecting over \(n\) gives \(t\mapsto t^2\in C_{\widehat\Phi^w,1}\). Since Example~\ref{ex:lacunary-weighted} gives \(\mathcal C_{\widehat\Phi^w,\tau}=\{t\mapsto t\}\), the two families cannot coincide.

In particular, the classical family contains \(t\mapsto t^2\), although \(\ell^2\) does not embed into \(\ell_M(w)\cong\ell^1\). The admissibility requirement excludes this function from \(\mathcal C_{\widehat\Phi^w,\tau}\). The universal-weight characterization in Theorem~\ref{th:weighted-Orlicz-characterization} does not apply to this example, since the weights are summable and hence \(w\notin\mathcal A\).
\end{remark}

	\section{Applications to Nakano Sequence Spaces}
	\label{sec:nakano}

\subsection{The admissible family for alternating exponents}
	
	Consider the Nakano sequence space with two alternating exponents introduced in Section~\ref{sec:known-nakano}. Proposition~\ref{prop:nakano-two-powers} gives its known embedding characterization. We now compute the associated admissible limiting family.
	
	\begin{example}\label{ex:nakano-limiting-family}
		Let \(1\leq p<q<\infty\), and let \(\Phi=\{\phi_n\}\), where \(\phi_n(t)=t^{p_n}\), \(p_n=p\) for odd \(n\), and \(p_n=q\) for even \(n\). Then
		\begin{equation}\label{eq:nakano-limiting-family}
			\mathcal C_{\Phi,\tau}
			=\left\{\alpha t^p+(1-\alpha)t^q:0\leq\alpha\leq1\right\},
		\end{equation}
		where the functions are considered on \([0,1]\).
	\end{example}
	
	\begin{proof}
		Fix \(n\geq n_0\), and let \(g\in\mathcal B_{\Phi,\tau,n}\). By definition, \(g(t)= \frac{\sum_{j\in E}\phi_j(\tau r_jt)} {\sum_{j\in E}\phi_j(\tau r_j)},\) where \(E\) is a nonempty finite subset of \(\{n,n+1,\ldots\}\), \(r_j>0\), and \(\sum_{j\in E}\phi_j(r_j)=1\). Put
		\[
		A_p=\tau^p\sum_{\substack{j\in E\\p_j=p}}r_j^p, \qquad A_q=\tau^q\sum_{\substack{j\in E\\p_j=q}}r_j^q.
		\]
		Then \(A_p+A_q>0\), and \(g(t)=\frac{A_p}{A_p+A_q}t^p +\frac{A_q}{A_p+A_q}t^q.\) Thus every element of \(\mathcal B_{\Phi,\tau,n}\) has the form \(g_\alpha(t)=\alpha t^p+(1-\alpha)t^q\) for some \(\alpha\in[0,1]\).
		
		Conversely, every \(\alpha\in[0,1]\) can be realized by an admissible block in the same tail. The endpoint \(\alpha=1\) is obtained by taking \(E=\{j_p\}\), where \(j_p\geq n\) is odd, and \(r_{j_p}=1\). Similarly, \(\alpha=0\) is obtained by taking \(E=\{j_q\}\), where \(j_q\geq n\) is even, and \(r_{j_q}=1\).
		
		For \(0<\alpha<1\), choose an odd index \(j_p\geq n\) and an even index \(j_q\geq n\), and set
		\[
		u=\frac{\alpha\tau^q} {(1-\alpha)\tau^p+\alpha\tau^q}, \qquad r_{j_p}=u^{1/p}, \qquad r_{j_q}=(1-u)^{1/q}.
		\]
		
		It follows that \(0<u<1\) and \(\phi_{j_p}(r_{j_p})+\phi_{j_q}(r_{j_q})=1\). The coefficient of \(t^p\) in the resulting normalized function is \(\frac{\tau^pu}{\tau^pu+\tau^q(1-u)}=\alpha.\) Consequently, for every \(n\geq n_0\), \({\mathcal B_{\Phi,\tau,n} =\left\{\alpha t^p+(1-\alpha)t^q:0\leq\alpha\leq1\right\}}.\) The set on the right is compact in \(C[0,1]\), being the continuous image of \([0,1]\) under \(\alpha\mapsto g_\alpha\). Taking closures and intersecting over the tails proves \eqref{eq:nakano-limiting-family}.
	\end{proof}
	
	\begin{remark}
		For this Nakano space, the \(\delta_2\)-data may be chosen as \(K=2^q\), \(\delta=1\), and \(c_n=0\). Since \(\frac{t\phi_n'(t)}{\phi_n(t)} = p_n,\) we have $\beta_{n,1}=p_n$ and $b_{\Phi,(c_n),1} = \limsup_{n\to\infty}p_n =q.$ Thus Theorem~\ref{th:upper-index} gives the necessary estimate \(s\leq q\) whenever \(\ell^s\subiso \ell^{(p_n)}\). The explicit description of \(\mathcal C_{\Phi,\tau}\) yields the sharper conclusion that only \(s=p\) and \(s=q\) are possible.
	\end{remark}
	
\subsection{The exact scalar cotype criterion}\label{sec:nakano-cotype}

We now consider an arbitrary sequence \((p_n)_{n=1}^{\infty}\) with \(1\leq p_n<\infty\), and write \(\ell^{(p_n)}=\ell_\Phi\), where \(\Phi=\{\phi_n\}_{n=1}^{\infty}\) and \(\phi_n(t)=t^{p_n}\). No boundedness of \((p_n)\) is assumed in this subsection. The next proposition recovers the scalar cotype criterion in \cite[Theorem~4.3]{BlascoGregori01}. We derive it from~\eqref{eq:cotype-characterization} by computing the least additive errors; the reduction to these errors is also recorded in \cite[Lemma~4.1(i)]{BlascoGregori01}.

\begin{proposition}[Blasco--Gregori]\label{prop:nakano-cotype-exact}
Let \(2\leq q<\infty\). The following conditions are equivalent:
\begin{enumerate}[label=(\roman*)]
\item The space \(\ell^{(p_n)}\) has cotype \(q\).
\item There exist \(C\geq1\) and a nonnegative sequence \((a_n)\in\ell^1\) such that
\begin{equation}\label{eq:nakano-cotype-errors}
t^q\leq Ct^{p_n}+a_n
\qquad(0\leq t\leq1,\ n\in\mathbb N).
\end{equation}
\item There exists \(\vartheta\in(0,1)\) such that
\begin{equation}\label{eq:nakano-cotype-series}
\sum_{\{n:\,p_n>q\}}\vartheta^{1/(p_n-q)}<\infty.
\end{equation}
\end{enumerate}
For fixed \(C\geq1\), the least possible error in~\eqref{eq:nakano-cotype-errors} is
\begin{equation}\label{eq:nakano-minimal-error}
d_n(q,C)=\max_{0\leq t\leq1}(t^q-Ct^{p_n})
=\begin{cases}
0,&p_n\leq q,\\[4pt]
\displaystyle
\left(1-\frac{q}{p_n}\right)
\left(\frac{q}{Cp_n}\right)^{q/(p_n-q)},&p_n>q.
\end{cases}
\end{equation}
Thus cotype \(q\) is also equivalent to \(\sum_n d_n(q,C)<\infty\) for some \(C\geq1\).
\end{proposition}

\begin{proof}
By~\eqref{eq:cotype-characterization}, condition~(i) is equivalent to \(\Phi\in\delta^q\). Suppose that~\eqref{eq:delta-power-cotype} holds with constants \(H\), \(\eta\), and errors \((e_n)\in\ell^1\). For \(0<s\leq1\), take \(v_n=\eta^{1/p_n}\), \(t=sv_n\), and \(\lambda=1/s\). Since \(\phi_n(\lambda t)=\eta\), we obtain
\[
\eta s^q\leq H\eta s^{p_n}+He_n.
\]
Hence~(ii) holds with \(C=H\) and \(a_n=He_n/\eta\).

Conversely, suppose~(ii) holds. For \(0<u\leq v\) with \(v^{p_n}\leq1\), apply~\eqref{eq:nakano-cotype-errors} at \(u/v\) and multiply by \(v^{p_n}\). This gives
\[
\left(\frac uv\right)^qv^{p_n}
\leq Cu^{p_n}+a_nv^{p_n}
\leq Cu^{p_n}+a_n.
\]
Taking \(u=t\) and \(v=\lambda t\) proves~\eqref{eq:delta-power-cotype} with \(H=C\), \(\eta=1\), and errors \(a_n/C\); the case \(t=0\) is immediate. Thus~(i) and~(ii) are equivalent.

For fixed \(C\), condition~(ii) holds exactly when the maxima \(d_n(q,C)\) form a summable sequence. If \(p_n\leq q\), then \(t^q\leq t^{p_n}\) on \([0,1]\), and the maximum is zero. If \(p_n>q\), differentiating \(t^q-Ct^{p_n}\) shows that its maximum occurs at
\[
t_n=\left(\frac{q}{Cp_n}\right)^{1/(p_n-q)}\in(0,1),
\]
which proves~\eqref{eq:nakano-minimal-error}.

To compare these errors with~\eqref{eq:nakano-cotype-series}, put \(a=p_n-q>0\) and \(x=q/a\). Then
\[
d_n(q,C)=\frac{1}{1+x}\left(1+\frac1x\right)^{-x}C^{-x}.
\]
The inequalities \((1+x)^{-1}\geq e^{-x}\) and \((1+1/x)^{-x}\geq e^{-1}\), together with the upper bound of one for both factors, give
\begin{equation}\label{eq:nakano-error-bounds}
e^{-1}(eC)^{-q/(p_n-q)}
\leq d_n(q,C)
\leq C^{-q/(p_n-q)}
\qquad(p_n>q).
\end{equation}
If the errors are summable, the left inequality yields~(iii) with \(\vartheta=(eC)^{-q}\). Conversely, given~(iii), choose \(C=\vartheta^{-1/q}>1\); the right inequality makes the errors summable. This proves all the equivalences.
\end{proof}

The exact criterion gives both the optimal exponent and a separate test for its attainment.

\begin{corollary}\label{cor:nakano-optimal-cotype}
Let \(L=\limsup_{n\to\infty}p_n\), with \(L=\infty\) allowed. Then
\begin{equation}\label{eq:nakano-optimal-cotype}
\operatorname{cot}(\ell^{(p_n)})=\max\{2,L\}.
\end{equation}
If \(L<2\), the space has cotype \(2\). If \(2\leq L<\infty\), the optimal exponent is attained if and only if
\begin{equation}\label{eq:nakano-cotype-endpoint}
\sum_{\{n:\,p_n>L\}}\vartheta^{1/(p_n-L)}<\infty
\quad\text{for some }\vartheta\in(0,1).
\end{equation}
If \(L=\infty\), the space has no finite cotype.
\end{corollary}

\begin{proof}
If \(q\geq2\) and \(q>L\), there are only finitely many indices with \(p_n>q\), so~\eqref{eq:nakano-cotype-series} holds. If \(2\leq q<L\), there is \(\varepsilon>0\) such that \(p_n\geq q+\varepsilon\) for infinitely many \(n\). For every \(\vartheta\in(0,1)\), the corresponding terms satisfy
\[
\vartheta^{1/(p_n-q)}\geq\vartheta^{1/\varepsilon}>0,
\]
so the series diverges. These observations prove~\eqref{eq:nakano-optimal-cotype}, including the case \(L=\infty\). When \(L<2\), the first observation applies to \(q=2\). When \(2\leq L<\infty\), Proposition~\ref{prop:nakano-cotype-exact} applied at \(q=L\) gives~\eqref{eq:nakano-cotype-endpoint}.
\end{proof}

For the alternating exponents in Example~\ref{ex:nakano-limiting-family}, this gives optimal cotype \(\max\{2,q\}\), which is attained because no exponent exceeds that value. More generally, when \(L<\infty\), one may choose \(c_n=0\), \(\delta=1\), and \(K=2^{\sup_n p_n}\) in~\eqref{eq:delta2}. The corrected upper index then equals \(b_{\Phi,(0),1}=L\). Corollary~\ref{cor:cotype-upper-index} gives cotype \(q\) for every \(q\geq2\) with \(q>L\); when \(L\geq2\),~\eqref{eq:nakano-cotype-endpoint} also decides the boundary case.

\begin{example}[Attainment of the optimal cotype]\label{ex:nakano-cotype-endpoint}
Fix \(q\geq2\) and consider the two sequences
\[
p_n^{(1)}=q+\frac{1}{\log(n+1)},
\qquad
p_n^{(2)}=q+\frac{1}{\sqrt{\log(n+1)}}.
\]
These examples are discussed in \cite[pp.~658--659]{BlascoGregori01}. Both spaces have optimal cotype \(q\), since both exponent sequences converge to \(q\). For the first sequence, the endpoint series is
\[
\sum_{n=1}^{\infty}\vartheta^{1/(p_n^{(1)}-q)}
=\sum_{n=1}^{\infty}(n+1)^{\log\vartheta},
\]
which converges for \(0<\vartheta<e^{-1}\). Thus \(\ell^{(p_n^{(1)})}\) has cotype \(q\). For the second sequence, writing \(c=-\log\vartheta>0\), the summands are
\[
\vartheta^{1/(p_n^{(2)}-q)}
=\exp\bigl(-c\sqrt{\log(n+1)}\bigr)
\geq\frac{1}{n+1}
\]
for all sufficiently large \(n\). The series diverges for every \(\vartheta\in(0,1)\). Hence \(\ell^{(p_n^{(2)})}\) has cotype \(q+\varepsilon\) for every \(\varepsilon>0\), but does not have cotype \(q\). The two spaces have the same upper index with zero errors, while their endpoint cotypes differ.
\end{example}
\par

\section{Cotype in Weighted Orlicz Sequence Spaces}\label{sec:weighted-cotype}

Throughout this section, \(M\) is an Orlicz function and \(w=(w_n)\) is a positive weight sequence. The following criterion is a direct reformulation of \cite[Theorem~9(a)]{Katirtzoglou98}, recalled in~\eqref{eq:cotype-characterization}. Cotype of weighted Orlicz sequence spaces was also studied in \cite{MaleevZlatanov00}. We give the details of the reformulation to make the dependence on both \(M\) and the weights explicit.

\subsection{An exact criterion for arbitrary positive weights}

For \(2\leq q<\infty\), \(C\geq1\), and \(\eta>0\), define
\begin{equation}\label{eq:weighted-cotype-defect}
D_n(q,C,\eta)
=w_n\sup_{\substack{0<u\leq v\\w_nM(v)\leq\eta}}
\left[\left(\frac{u}{v}\right)^qM(v)-CM(u)\right]_+,
\end{equation}
where \([a]_+=\max\{a,0\}\). Equivalently, the supremum is taken over \(0<u\leq v\leq M^{-1}(\eta/w_n)\). Each \(D_n(q,C,\eta)\) is finite, since \(0\leq D_n(q,C,\eta)\leq\eta\).

\begin{proposition}\label{prop:weighted-cotype-defect}
Let \(2\leq q<\infty\). Then
\(\ell_M(w)\) has cotype \(q\)
if and only if there exist \(C\geq 1\) and \(\eta>0\) such that
\begin{equation}\label{eq:weighted-cotype-test}
\sum_{n=1}^{\infty}D_n(q,C,\eta)<\infty.
\end{equation}

Consequently,
\begin{equation}\label{eq:weighted-optimal-cotype}
\operatorname{cot}(\ell_M(w))
=\inf\left\{q\geq2:\exists C\geq1,\ \eta>0,
\ \sum_{n=1}^{\infty}D_n(q,C,\eta)<\infty\right\}.
\end{equation}
\end{proposition}

\begin{proof}
Use the generators \(\phi_n^w(t)=w_nM(t)\) from~\eqref{eq:weighted-musielak-function}. Suppose first that \(\ell_M(w)\) has cotype \(q\). By~\eqref{eq:cotype-characterization}, \(\Phi^w\in\delta^q\). Choose \(H\geq1\), \(\eta>0\), and a nonnegative sequence \((d_n)\in\ell^1\) satisfying~\eqref{eq:delta-power-cotype}. For \(0<u\leq v\) with \(w_nM(v)\leq\eta\), take \(t=u\) and \(\lambda=v/u\) to obtain
\[
w_n\left[\left(\frac{u}{v}\right)^qM(v)-HM(u)\right]_+
\leq H d_n.
\]
Thus \(D_n(q,H,\eta)\leq H d_n\), proving the required summability.

Conversely, let \(C\) and \(\eta\) satisfy~\eqref{eq:weighted-cotype-test}. Put \(d_n=D_n(q,C,\eta)/C\). If \(t>0\), \(\lambda\geq1\), and \(w_nM(\lambda t)\leq\eta\), the definition of \(D_n\), applied with \(u=t\) and \(v=\lambda t\), gives
\[
\lambda^{-q}w_nM(\lambda t)-Cw_nM(t)
\leq D_n(q,C,\eta).
\]
Therefore
\[
w_nM(\lambda t)
\leq C\lambda^q\bigl(w_nM(t)+d_n\bigr).
\]
The case \(t=0\) is immediate. Hence \(\Phi^w\in\delta^q\), and~\eqref{eq:cotype-characterization} gives cotype \(q\). Formula~\eqref{eq:weighted-optimal-cotype} follows by taking the infimum.
\end{proof}

The same criterion applies to \(\widehat\phi_n^w(t)=w_nM(s_nt)\). The substitutions \(u=s_nx\) and \(v=s_ny\) leave the supremum in~\eqref{eq:weighted-cotype-defect} unchanged. Thus the criterion is consistent with the diagonal isometry used in the embedding results. It includes the endpoint \(q=2\) and requires only positivity of the weights and the defining assumptions on \(M\).
	
	\begin{example}[A prescribed cotype without \(\Delta_2\) at zero]
		\label{ex:weighted-cotype-no-delta2}
		Fix \(p>2\), and put
		\[
		a_n=2^{-2^n},
		\qquad
		w_n=a_n^{-(p+1)}=2^{(p+1)2^n},
		\qquad n\geq1.
		\]
		Define
		\[
		A(t)=
		\begin{cases}
			0,&t=0,\\[2mm]
			(a_n+a_{n+1})t-a_na_{n+1},
			&a_{n+1}\leq t\leq a_n,\quad n\geq1,\\[2mm]
			t^2,&t\geq a_1,
		\end{cases}
		\]
		and let \(M_p(t)=t^{p-1}A(t)\).
		Then \(M_p\) is an Orlicz function that fails the
		\(\Delta_2\) condition at zero, whereas
		\[
		\ell_{M_p}(w)=h_{M_p}(w),
		\qquad
		\operatorname{cot}\bigl(\ell_{M_p}(w)\bigr)=p.
		\]
		The space has cotype \(p\), so the optimal exponent is attained.
	\end{example}
	
	\begin{proof}
		We have \(a_1=1/4\), \(a_{n+1}=a_n^2\), and
		\(\sum_n a_n<\infty\).
		The function \(A\) is continuous and convex:
		its linear pieces join the points \((a_n,a_n^2)\),
		and their slopes \(a_n+a_{n+1}\) increase as the
		argument increases. At \(a_1\), the left derivative
		is at most the right derivative \(2a_1\).
		Since \(t^{p-1}\) and \(A(t)\) are nonnegative,
		nondecreasing, convex functions, their product \(M_p\)
		is convex. Thus \(M_p\) is an Orlicz function.
		
		For every \(n\), we have
		\[
		\frac{M_p(2a_{n+1})}{M_p(a_{n+1})}
		=
		2^{p-1}\frac{A(2a_{n+1})}{A(a_{n+1})}
		=
		2^{p-1}\left(\frac1{a_n}+2\right)
		\longrightarrow\infty.
		\]
		Hence \(M_p\) does not satisfy the \(\Delta_2\)
		condition at zero.
		
		We first estimate the normalized generators.
		Convexity and the supporting line on
		\([a_{n+1},a_n]\) give
		\[
		a_nt-a_n^3\leq A(t)\leq a_nt,
		\qquad 0\leq t\leq a_n.
		\]
		Consequently, the functions
		\[
		\varphi_n(t)=\frac{A(a_nt)}{a_n^2}
		\]
		satisfy
		\begin{equation}\label{eq:prescribed-cotype-linear-error}
			0\leq t-\varphi_n(t)\leq a_n,
			\qquad 0\leq t\leq1.
		\end{equation}
		Since \(M_p(a_n)=a_n^{p+1}=1/w_n\), the normalized
		generators of \(\ell_{M_p}(w)\) are
		\[
		\psi_n(t)=w_nM_p(a_nt)=t^{p-1}\varphi_n(t).
		\]
		It follows that
		\begin{equation}\label{eq:prescribed-cotype-power-error}
			0\leq t^p-\psi_n(t)
			\leq a_nt^{p-1}\leq a_n,
			\qquad 0\leq t\leq1.
		\end{equation}
		
		We next verify the \(\delta_2\) condition.
		By~\eqref{eq:prescribed-cotype-linear-error},
		\[
		\psi_n(1/2)
		\geq 2^{-(p-1)}(1/2-a_n)
		\geq 2^{-(p+1)}.
		\]
		Thus, if \(\psi_n(t)\leq\delta=2^{-(p+1)}\),
		then \(t\leq1/2\).
		For such \(t\), estimate
		\eqref{eq:prescribed-cotype-linear-error} yields
		\[
		\varphi_n(2t)\leq2t
		\leq2\varphi_n(t)+2a_n.
		\]
		Therefore
		\[
			\psi_n(2t)
			=(2t)^{p-1}\varphi_n(2t)
			\leq2^p\psi_n(t)+2^pa_nt^{p-1}
			\leq2^p\psi_n(t)+2a_n.
		\]
		The normalized generators satisfy \(\delta_2\)
		with \(K=2^p\), \(\delta=2^{-(p+1)}\), and
		summable errors \(c_n=2a_n\).
		Replacing \(t\) by \(u/a_n\) gives the same condition
		for the original generators \(\{w_nM_p\}\).
		Consequently,
		\(\ell_{M_p}(w)=h_{M_p}(w)\).
		
		We now compute the cotype using
		Proposition~\ref{prop:weighted-cotype-defect}.
		After the substitutions \(u=a_nx\), \(v=a_ny\),
		the defect takes the form
		\[
		D_n(q,C,\eta)
		=
		\sup_{\substack{0<x\leq y\\\psi_n(y)\leq\eta}}
		\left[
		\left(\frac{x}{y}\right)^q\psi_n(y)-C\psi_n(x)
		\right]_+.
		\]
		For \(q=p\), \(C=1\), and \(\eta=1\),
		normalization gives \(y\leq1\).
		Using~\eqref{eq:prescribed-cotype-power-error}, we obtain
		\[
		\left(\frac{x}{y}\right)^p\psi_n(y)-\psi_n(x)
		\leq x^p-\psi_n(x)
		\leq a_n.
		\]
		
		Hence
		\(D_n(p,1,1)\leq a_n\) and \(\sum_{n=1}^{\infty}D_n(p,1,1)<\infty\).
		Proposition~\ref{prop:weighted-cotype-defect}
		therefore gives cotype \(p\).
		
		To prove optimality, fix \(2\leq q<p\),
		\(C\geq1\), and \(\eta>0\).
		Choose
		\[
		y=\min\{1,\eta^{1/p}\},
		\qquad
		z=(2C)^{-1/(p-q)}.
		\]
		Then \(0<z<1\), \(\psi_n(y)\leq y^p\leq\eta\),
		and \(Cz^p=z^q/2\).
		Evaluating the defect at \(x=zy\) and using
		\eqref{eq:prescribed-cotype-power-error}, we obtain
		\[
			D_n(q,C,\eta)
			\geq z^q\psi_n(y)-C\psi_n(zy)
			\geq z^q(y^p-a_n)-Cz^py^p
			=\frac12y^pz^q-z^qa_n
			\geq\frac12y^pz^q-a_n.
		\]
		The positive constant \(c=\frac12y^pz^q\)
		is independent of \(n\).
		Since \(a_n\to0\), we have
		\(D_n(q,C,\eta)\geq c/2\) for all sufficiently
		large \(n\). Thus
		\(\sum_{n=1}^{\infty}D_n(q,C,\eta)=\infty\).
		
		As \(C\) and \(\eta\) were arbitrary,
		Proposition~\ref{prop:weighted-cotype-defect}
		excludes cotype \(q\) for every \(2\leq q<p\).
		
		Therefore
		\(\operatorname{cot}\bigl(\ell_{M_p}(w)\bigr)=p\),
		and this exponent is attained.
	\end{proof}
	
	\par

\subsection{The formula for weights in \(\mathcal A\)}

Define the global upper growth index by
\begin{equation}\label{eq:global-upper-index}
\beta_M^{\mathrm{glob}}
=\inf\left\{r\geq1:
\begin{array}{l}
\text{there exists }C\geq1\text{ such that}\\
M(\lambda t)\leq C\lambda^rM(t)
\text{ for all }t>0,\ \lambda\geq1
\end{array}\right\},
\end{equation}
with \(\inf\varnothing=\infty\). For weights in \(\mathcal A\), the summable errors in Proposition~\ref{prop:weighted-cotype-defect} can be eliminated. This gives the global growth characterization directly from the known Musielak--Orlicz cotype criterion.

\begin{corollary}\label{cor:weighted-cotype-A}
Let \(w\in\mathcal A\) and \(2\leq q<\infty\). Then \(\ell_M(w)\) has cotype \(q\) if and only if there exists \(C\geq1\) such that
\begin{equation}\label{eq:weighted-global-cotype}
M(\lambda t)\leq C\lambda^qM(t)
\qquad(t>0,\ \lambda\geq1).
\end{equation}
In particular,
\begin{equation}\label{eq:weighted-cotype-global-index}
\operatorname{cot}(\ell_M(w))
=\max\{2,\beta_M^{\mathrm{glob}}\}.
\end{equation}
If \(\beta_M^{\mathrm{glob}}<2\), the space has cotype \(2\). If \(2\leq\beta_M^{\mathrm{glob}}<\infty\), it has cotype \(\beta_M^{\mathrm{glob}}\) exactly when~\eqref{eq:weighted-global-cotype} holds with \(q=\beta_M^{\mathrm{glob}}\).
\end{corollary}

\begin{proof}
Suppose that \(\ell_M(w)\) has cotype \(q\), and choose \(C\) and \(\eta\) as in Proposition~\ref{prop:weighted-cotype-defect}. By the definition of \(\mathcal A\), there is a subsequence \((w_{n_k})\) tending to zero with \(\sum_kw_{n_k}=\infty\). Fix \(0<u\leq v\). If
\[
B=\left(\frac{u}{v}\right)^qM(v)-CM(u)>0,
\]
then \(w_{n_k}M(v)\leq\eta\) for all sufficiently large \(k\), and consequently
\[
D_{n_k}(q,C,\eta)\geq B w_{n_k}.
\]
This contradicts \(\sum_nD_n(q,C,\eta)<\infty\). Hence \(B\leq0\) for every \(0<u\leq v\), which is~\eqref{eq:weighted-global-cotype} with \(t=u\) and \(\lambda=v/u\).

Conversely,~\eqref{eq:weighted-global-cotype} makes every term \(D_n(q,C,\eta)\) zero, for any \(\eta>0\). Proposition~\ref{prop:weighted-cotype-defect} then gives cotype \(q\).

The set of exponents in~\eqref{eq:global-upper-index} is upward closed. If \(\beta_M^{\mathrm{glob}}<\infty\), every \(q>\beta_M^{\mathrm{glob}}\) belongs to that set, and every member is at least \(\beta_M^{\mathrm{glob}}\). Taking the infimum over \(q\geq2\) proves~\eqref{eq:weighted-cotype-global-index}, and also gives the stated endpoint conclusions. If \(\beta_M^{\mathrm{glob}}=\infty\), the equivalence just proved excludes every finite cotype, so the same formula holds.
\end{proof}

In particular, for \(w\in\mathcal A\) a finite cotype exists exactly when \(M\) satisfies the global \(\Delta_2\) condition. Indeed, a global power estimate implies doubling, while iteration of a doubling estimate gives a global power estimate. Thus no separate global \(\Delta_2\) assumption is needed in Corollary~\ref{cor:weighted-cotype-A}.

\subsection{Explicit computations}

For \(M(t)=t^p\), where \(1\leq p<\infty\), the criterion can be evaluated for every positive weight sequence. Writing \(z=u/v\) and using \(w_nv^p\leq\eta\), we obtain
\begin{equation}\label{eq:power-cotype-defect}
D_n(q,C,\eta)=
\begin{cases}
\displaystyle
\eta\left(1-\frac qp\right)
\left(\frac{q}{Cp}\right)^{q/(p-q)},&2\leq q<p,\\[6pt]
0,&q\geq\max\{2,p\}.
\end{cases}
\end{equation}
For \(q<p\), the maximum of \(z^q-Cz^p\) over \([0,1]\) occurs at
\[
z=\left(\frac{q}{Cp}\right)^{1/(p-q)}.
\]
We may also choose \(w_nv^p=\eta\). Hence the positive expression in~\eqref{eq:power-cotype-defect} is independent of \(n\), and its series diverges. For \(q\geq p\), one has \(z^q\leq z^p\leq Cz^p\), so every term is zero. Proposition~\ref{prop:weighted-cotype-defect} therefore gives
\[
\operatorname{cot}(\ell_{t^p}(w))=\max\{2,p\},
\]
and this cotype is attained for every positive \(w\).

The next computation shows why the weight condition matters when the growth of \(M\) changes between zero and infinity.

\begin{example}\label{ex:weighted-cotype-dependence}
Fix \(p>2\), and put
\[
M_p(t)=
\begin{cases}
t^p,&0\leq t\leq1,\\
pt-(p-1),&t>1.
\end{cases}
\]
Then \(M_p\) is a continuously differentiable convex Orlicz function, and
\[
\operatorname{cot}(\ell_{M_p}(w))=
\begin{cases}
p,&w\in\mathcal A,\\
2,&\sum_nw_n<\infty.
\end{cases}
\]
Both displayed cotypes are attained. In particular, the same function \(M_p\) gives optimal cotype \(p\) for \(w_n=1/n\), and optimal cotype \(2\) for \(w_n=2^{-n}\).
\end{example}

\begin{proof}
The values and first derivatives match at \(t=1\), and the derivative is nondecreasing. Moreover,
\[
\frac{tM_p'(t)}{M_p(t)}
=\begin{cases}
p,&0<t\leq1,\\
\displaystyle\frac{pt}{pt-(p-1)},&t>1,
\end{cases}
\]
so this quotient is at most \(p\). Integration gives \(M_p(\lambda t)\leq\lambda^pM_p(t)\) for \(\lambda\geq1\). No global upper growth exponent smaller than \(p\) is possible: for \(0<u\leq1\), a bound \(M_p(1)\leq C u^{-r}M_p(u)\) with \(r<p\) would give \(1\leq C u^{p-r}\), a contradiction as \(u\to0\). Thus \(\beta_{M_p}^{\mathrm{glob}}=p\), with the bound attained, and Corollary~\ref{cor:weighted-cotype-A} applies when \(w\in\mathcal A\).

For summable weights, observe that
\[
pt-(p-1)\leq M_p(t)\leq pt\qquad(t\geq0).
\]
Therefore, for \(0<u\leq v\),
\[
\left(\frac uv\right)^2M_p(v)-M_p(u)
\leq pu-M_p(u)\leq p-1.
\]
It follows that \(D_n(2,1,\eta)\leq(p-1)w_n\) for every \(\eta>0\). The sum is finite, so Proposition~\ref{prop:weighted-cotype-defect} gives cotype \(2\), and hence optimal cotype \(2\).
\end{proof}

The summable-weight argument also applies at \(p=2\), recovering cotype \(2\) for the lacunary weighted space in Example~\ref{ex:lacunary-weighted}. In that example the conclusion also follows from \(b_{\widehat\Phi^w,(w_n),1}=1\) and Corollary~\ref{cor:cotype-upper-index}.
\par

	\section{Relation to Earlier Results}
	\label{sec:comparison-peirats-ruiz}
	
We compare the admissible-block construction and the corrected upper index with the limiting-family approach of Peirats and Ruiz \cite{Peiratz92} and the weighted-space methods of Fuentes and Hern\'andez \cite{FuentesHernandez88}. The comparison concerns growth assumptions, the choice of generators, and the role of summable errors.
	
	\subsection{Growth assumptions and the underlying spaces}
	
	Peirats and Ruiz work with a class that includes nonconvex Orlicz functions and locally bounded spaces. Their standing assumptions include uniform upper and lower bounds for the logarithmic derivatives of the generating functions. Written in our notation, these bounds have the form \(0<a\leq\frac{t\phi_n'(t)}{\phi_n(t)}\leq b<\infty, 0<t\leq1,\) with constants independent of \(n\).
	
In \cite{Peiratz92}, the uniform \(\Delta_2^*\) condition denotes the lower logarithmic-derivative bound above. For the convex differentiable generators considered here, \(t\phi_n'(t)\geq\phi_n(t)\), so a uniform lower bound of one is automatic.
	
Zlatanov \cite{Zlatanov07} gave a sufficient condition for the isomorphic embedding of \(\ell^p\)-spaces into Musielak--Orlicz sequence spaces. Theorem~\ref{th:upper-index} instead gives a necessary restriction for the specified generators under \(\delta_2\), with an additive summable-error term. The two criteria address different directions of the embedding question; neither yields the converse of the other in the present setting.

	The present paper is restricted to convex Orlicz functions and Banach sequence spaces. Within this setting, the main construction assumes the \(\delta_2\) estimate \({\phi_n(2t)\leq K\phi_n(t)+c_n}\) whenever \(\phi_n(t)\leq\delta\), where \((c_n)\) is nonnegative and summable. The compactness and block-sequence arguments apply directly to the given functions and do not require their differentiability.
	
	This distinction concerns the chosen representation of the space by its generating functions. Peirats and Ruiz also recall results of Woo concerning equivalent representations of separable convex modular sequence spaces by functions satisfying uniform growth conditions. Accordingly, the use of \(\delta_2\) here should be understood as providing direct control of the specified generators and their summable errors.
	
	\subsection{Classical and admissible limiting families}
	
	The classical family \(C_{\Phi,1}\), introduced in \cite{Peiratz92}, is obtained from finite convex combinations of normalized dilations over successive tails. Proposition~2.3 of that paper shows that if a space with a symmetric basis embeds into \(\ell_\Phi\), then it admits an Orlicz representation generated by an element of \(C_{\Phi,1}\).
	
	Our family \(\mathcal C_{\Phi,\tau}\) incorporates an additional admissibility requirement. Its approximating functions have the form
	\[
	g(t)=\frac{1}{A}\sum_{j\in E}\phi_j(\tau r_jt), \qquad A=\sum_{j\in E}\phi_j(\tau r_j),
	\]
	where \(\sum_{j\in E}\phi_j(r_j)=1\). Thus the coefficients in the corresponding convex combination are determined by the coordinates of a normalized block. Lemma~\ref{lm:admissible-inclusion} establishes the relation \(\mathcal C_{\Phi,\tau}\subseteq C_{\Phi,1}\). Remark~\ref{rem:strict-inclusion} shows that this inclusion can be strict: in the lacunary weighted example, the admissible family consists only of \(t\mapsto t\), whereas the classical family also contains \(t\mapsto t^2\).
	
	The estimate \(A\geq a>0\) in Lemma~\ref{lm:normalizing-constants} is essential when summable errors are present. It ensures that the normalized error terms are bounded by \(a^{-1}\sum_{j\geq n}c_j\), which tends to zero along the tails. This leads to the compactness result in Lemma~\ref{lm:compactness} and to the representation in Theorem~\ref{prop:1}, where the defining function belongs to \(\mathcal C_{\Phi,\tau}\).
	
	The reduction to blocks and the extraction of uniformly convergent subsequences follow the strategy of \cite{Peiratz92}. Lemma~\ref{lm:summable-uniform-perturbations} makes explicit the norm comparison underlying the summable approximation step used in that strategy. For the converse in Theorem~\ref{prop:1}\ref{th:1.2}, membership in every tail closure permits successive admissible blocks to approximate any \(\varphi\in\mathcal C_{\Phi,\tau}\) with summable uniform errors. The same norm comparison then constructs a copy of \(\ell_\varphi\) in \(\ell_\Phi\), giving the characterization by the admissible family.
	
	\subsection{Comparison of the upper indices}
	
	Proposition~2.4 of \cite{Peiratz92} gives lower and upper restrictions on the exponent of an embedded \(\ell^p\) space. In our notation, their upper index is \(b_\Phi^{\mathrm{PR}} =\limsup_{n\to\infty} \sup_{0<t\leq1} \frac{t\phi_n'(t)}{\phi_n(t)}.\) The index used in Theorem~\ref{th:upper-index} is instead \(b_{\Phi,(c_n),\delta} =\limsup_{n\to\infty} \sup_{0<t\leq\gamma_\delta} \frac{t\phi_n'(t)-c_n}{\phi_n(t)}.\) It is defined with respect to fixed \(\delta_2\)-data and includes the additive errors from the growth estimate.
	
	When \(c_n=0\) for every \(n\) and \(\delta=1\), normalization gives \(\gamma_1=1\), and the two upper indices coincide. In this case, Theorem~\ref{th:upper-index} recovers the upper restriction in \cite{Peiratz92}. For nonzero summable errors, the proof uses the corrected quotients \((\psi_k(t)+d_k)/t^r\), where \(d_k\to0\). The lower bound for the normalizing constants ensures that these corrections disappear in the limit. The resulting theorem addresses the upper restriction on \(p\); the two-sided estimate in \cite{Peiratz92} additionally uses a lower growth index.
	
	Example~\ref{ex:nonuniform-growth} illustrates the distinction between these indices. For its normalized generating sequence, \(\sup_{0<t\leq1}t\phi_n'(t)/\phi_n(t)=p_n\to\infty\), so the uniform upper growth assumption of \cite{Peiratz92} fails. For the specified summable error sequence, however, \(b_{\Phi,(c_n),1}=1\), and Theorem~\ref{th:upper-index} gives the exact restriction on the exponents of embedded \(\ell^p\) spaces. The example concerns direct applicability to the given generators: the resulting space is isomorphic to \(\ell^1\), which also admits a representation satisfying the uniform growth assumptions.
	
Example~\ref{ex:lacunary-weighted} also provides a quantitative comparison of the two upper indices for the same normalized generators. With \(u_n=2w_n/(1+w_n)\), their uncorrected logarithmic derivatives satisfy
\[
\frac{t(\widehat\phi_n^w)'(t)}{\widehat\phi_n^w(t)}
=
\begin{cases}
2,&0<t\leq u_n,\\[4pt]
\displaystyle\frac{(1+w_n)t}{(1+w_n)t-w_n},&u_n<t\leq1.
\end{cases}
\]
The second expression decreases from \(2\) as \(t\) increases from \(u_n\). Hence \(\sup_{0<t\leq1}t(\widehat\phi_n^w)'(t)/\widehat\phi_n^w(t)=2\) for every \(n\), so \(b_{\widehat\Phi^w}^{\mathrm{PR}}=2\). For the summable errors \(c_n=w_n\), Example~\ref{ex:lacunary-weighted} gives \({b_{\widehat\Phi^w,(w_n),1}}=1\). Thus the corrected index sharpens the numerical necessary bound from \(s\leq2\) to \(s\leq1\), consistently with the admissible family \(\{t\mapsto t\}\).
\par

	\subsection{Weighted Orlicz and Nakano spaces}
	
	The disjoint representation in Theorem~\ref{th:LMinfty} and the weighted characterization in Theorem~\ref{th:weighted-Orlicz-characterization} are included as known results corresponding to \cite{Peiratz92}. They connect the admissible-family construction with the classical theory of Orlicz function spaces and universal weight sequences. Theorem~\ref{th:weighted-Orlicz-copy-necessary} applies the equivalence in Theorem~\ref{prop:1}\ref{th:1.2} to the normalized weighted generators and yields the corresponding embedding characterization.
	
The variable-exponent results of Hern\'andez and Ruiz \cite{HernandezRuiz12} concern, in particular, lattice-isomorphic copies in Nakano function spaces. The recent study by Ansorena and Bello \cite{AnsorenaBello26} focuses on subsymmetric and complemented basic sequences in variable-exponent Lebesgue spaces. Our alternating-exponent example computes the admissible family for a sequence space and concerns isomorphic embeddings without a complementability hypothesis; its sharper exponent restriction is already known from Peirats and Ruiz \cite{Peiratz92}.

	For Nakano spaces, \cite{Peiratz92} gives a general characterization: under boundedness of the exponents away from zero and infinity, \(\ell^s\) embeds if and only if \(s\) is the limit of a subsequence of the exponent sequence. Proposition~\ref{prop:nakano-two-powers} is its special case for two alternating exponents. Example~\ref{ex:nakano-limiting-family} complements this known characterization by explicitly computing the admissible family \(\mathcal C_{\Phi,\tau}\) and illustrating its role in the present construction.

The cotype computation in Section~\ref{sec:nakano-cotype} recovers the scalar criterion of Blasco and Gregori \cite[Theorem~4.3]{BlascoGregori01} from the Musielak--Orlicz characterization of Katirtzoglou \cite[Theorem~9(a)]{Katirtzoglou98}. Its explicit minimal errors distinguish attainment of the optimal cotype from the value of its infimum. The criterion is included as a known result with a direct computation.
	
	\subsection{Comparison with Fuentes and Hern\'andez}
	
Hern\'andez and Peirats \cite{HernandezPeirats88} studied \(\ell^p\)-isomorphic weighted sequence subspaces of Orlicz function spaces. Our weighted result, Theorem~\ref{th:weighted-Orlicz-copy-necessary}, characterizes arbitrary isomorphic embeddings of Orlicz sequence spaces into \(\ell_M(w)\), assuming that the normalized weighted generators satisfy \(\delta_2\). The criterion is equivalence at zero to a member of the admissible family \(\mathcal C_{\widehat\Phi^w,\tau}\).

	Fuentes and Hern\'andez \cite{FuentesHernandez88} study weighted Orlicz sequence spaces, with particular emphasis on summable weights and the relation between growth at infinity and the geometry of the resulting spaces. Their results concern natural inclusions, separability, copies of \(\ell^\infty\) and \(c_0\), reflexivity, and \(B\)-convexity. These questions complement the study of embedded Orlicz sequence spaces and finite exponents undertaken here.
	
	For summable weights satisfying \eqref{eq:FH-weight-condition}, Fuentes and Hern\'andez \cite[Theorem~3]{FuentesHernandez88} characterize separability, equality with \(h_M(w)\), and the absence of an isomorphic copy of \(\ell^\infty\) through the \(\Delta_2\) condition for \(M\) at infinity. Under the same condition on the weights and the assumptions on the complementary function stated there, Proposition~6 characterizes reflexivity and related geometric properties. Our embedding characterization assumes \(\widehat\Phi^w\in\delta_2\) directly and does not impose this condition on the weights. The conclusions are different, however, so the present results should not be viewed as a generalization of those geometric characterizations.
	
	The normalization \(\widehat\phi_n^w(t)=w_nM(M^{-1}(1/w_n)t)\) already appears in the proof of \cite{FuentesHernandez88}. We use the same diagonal identification to apply the admissible-family construction to weighted spaces. Also, Proposition~1 of that paper concerns the natural inclusion between two weighted sequence spaces with the same weights. This should be distinguished from the existence of an arbitrary isomorphic embedding considered in Theorem~\ref{th:weighted-Orlicz-copy-necessary}.
	
The summable approximation method of \cite[Examples~1 and~2]{FuentesHernandez88} underlies Lemma~\ref{lm:summable-uniform-perturbations}, which treats approximation by an arbitrary normalized Orlicz function. Lemma~\ref{lm:power-singleton-family} additionally identifies the admissible family as \(\{t\mapsto t^p\}\), giving \(\{t\mapsto t\}\) and \(\{t\mapsto t^2\}\) for those examples.
	
	Fuentes and Hern\'andez explicitly discuss the failure of uniform growth conditions to be preserved under equivalent representations of a modular sequence space. Thus the absence of uniform upper growth for a given generating sequence does not by itself distinguish our results from theirs. The specific role of the summable errors in the present work is to control admissible normalized blocks and to define the corrected upper index.
	
Example~\ref{ex:lacunary-weighted} illustrates the corrected index for summable weights that fail \eqref{eq:FH-weight-condition}. Its identification with \(\ell^1\) uses the summable approximation method of \cite{FuentesHernandez88}, whereas the additional conclusions concern the corrected index and the admissible limiting family.
	
	\section{Conclusion}
Under \(\delta_2\), the admissible limiting family represents embedded symmetric bases and characterizes embedded Orlicz sequence spaces through equivalence at zero. The converse is realized by successive admissible blocks, and diagonal normalization gives the weighted characterization. Under the stated differentiability assumptions, the corrected upper index gives the necessary bound \(p\leq b_{\Phi,(c_n),\delta}\). The same index also yields cotype \(q\) for every \(q\geq2\) with \(q>b_{\Phi,(c_n),\delta}\); in particular, \(b_{\Phi,(c_n),\delta}<2\) implies cotype \(2\). The embedding bound is exact in the two examples isomorphic to \(\ell^1\), while the alternating Nakano example shows that the family criterion can be sharper. The construction builds on the methods of Peirats and Ruiz \cite{Peiratz92} and the summable approximation method of Fuentes and Hern\'andez \cite{FuentesHernandez88}. The index records \(\Phi\), \((c_n)\), and \(\delta\). For Nakano spaces, the exact summability criterion recovers \(\operatorname{cot}(\ell^{(p_n)})=\max\{2,\limsup_n p_n\}\) and determines whether this infimum is attained. For weighted spaces, the cotype criterion takes the form of a summability test; for \(w\in\mathcal A\), it reduces to \(\operatorname{cot}(\ell_M(w))=\max\{2,\beta_M^{\mathrm{glob}}\}\). The explicit computations show that the weights can change the optimal cotype when the growth at zero differs from the growth at infinity. The family criterion is necessary and sufficient; no sufficiency for an \(\ell^p\)-embedding is asserted for the index inequality alone.

\end{document}